\documentclass[a4paper, final]{siamart251216} 

\usepackage{amsmath, latexsym, amsfonts, amssymb, amscd,bbm , stmaryrd}
\usepackage[english]{babel}
\usepackage{mathrsfs}
\usepackage{damacros}

\usepackage{color}

\newcommand{\norm}[1]{\left \| #1 \right \|}

 \newcommand{\lsup}[1]{\underset{#1\to\infty}{\overline{\lim}}}

\title{McKean-Vlasov Equations for Large Networks of Neurons with Adaptive Asymmetric Edges}

\author{%
  Daniele Avitabile%
  \thanks{%
    Amsterdam Centre for Dynamics and Computation,
    Vrije Universiteit Amsterdam,
    Department of Mathematics,
    Faculteit der Exacte Wetenschappen,
    De Boelelaan 1081a,
    1081 HV Amsterdam, The Netherlands.
  \protect
    MathNeuro Team,
    Inria branch of the University of Montpellier,
    860 rue Saint-Priest
    34095 Montpellier Cedex 5
    France.
  \protect
  (\email{d.avitabile@vu.nl}, \url{www.danieleavitabile.com}, \url{www.amsterdam-dynamics.nl}).
  }
  \and
  James MacLaurin
  \thanks{Department of Mathematical Sciences. 
    New Jersey Institute of Technology, \email{james.n.maclaurin@njit.edu}
  }
}

\begin{document}

\maketitle

\begin{abstract}
Classical Mckean-Vlasov theory concerns high-dimensional particle systems for which
the effect of one particle on another is instantaneous. However in many applications,
particularly neuroscience, the effect of one particle on another is not instantaneous
but delayed. For biophysically-accurate neuroscientific models, for each $j\mapsto k$ connection
one can introduce an additional stochastic process that describes the propagation of
the signal from $j$ to $k$. In this paper we determine a self-consistent autonomous
SDE that describe the high-dimensional limit of such neural networks.  We apply this
system to a model of the visual cortex from mathematical neuroscience. We determine the existence of spatio-temporal oscillations in the average neural activity.
\end{abstract}

\section{Introduction}
Classically, high-dimensional particle systems are usually assumed to be (i)
exchangeable and (ii) with instantaneous all-to-all-interactions, as outlined in the
seminal works of Kac \cite{Kac1956}, Mckean \cite{mckean1956elementary}, Tanaka
\cite{Tanaka1982}, Sznitman \cite{Sznitman1989} and others
\cite{jabin2017mean,avitabile2026neural}. However in many applications it is increasingly
recognized that high-dimensional networks are more complicated than was envisaged by
the classical theory
\cite{delattre2016hawkes,lacker2023local,ocker2023republished,maclaurin2024hydrodynamic,jabin2025mean}. In this paper we
study networks such that both the node variables and the edge variables satisfy
stochastic differential equations. A paradigmatic example lies in neuroscience,
because the process by which a neuron sends an electrical impulse to another neuron
is complex, delayed and stochastic
\cite{Ermentrout2010,rosenbaum2012short,coombes2023neurodynamics,ocker2023republished,maclaurin2026large}. Another popular application of adaptive networks is in epidemiology \cite{baldassarri2026infection}.

In the literature, these systems are often referred to as Adaptive Networks or Co-evolutionary Networks,
consisting of node variables and edge variables that evolve stochastically in time
\cite{Kuehn2012,Berner2023}. There is a burgeoning literature on adaptive networks:
see the recent surveys in \cite{sawicki2023perspectives,Berner2023}. It is known that
adaptive networks can support a huge range of interesting patterns and phase
transitions, especially when the edge-adaptation has a very different timescale to
the node dynamics
\cite{yanchuk2009delay,yanchuk2017spatio,Lucon2018,balzer2024canard}. Network
adaptation can also lead to a loss of function \cite{zeng2025bursty}. Other works
have explored how network adaptation can be a means of controlling macroscopic
phenomena like synchronization \cite{lehnert2014controlling}.

There exist methods for characterizing the population density of high-dimensional
adaptive networks
\cite{MacLaurinAdaptive2024,gkogkas2025mean,greven2025continuum,athreya2025co,carrillo2025evolution,maclaurin2026large,ganguly2026mean}.
Generally, even if the original particle system is autonomous, the limiting equations
are non-autonomous. The basic reason is that, to obtain population density equations
to characterize the large $n$ limit, the states of the $O(n^2)$ edges must
be approximated by `their average, conditioned on the path-history of the node
activity at either end'. This leads to a self-consistent delay stochastic
differential equation. We proved this self-consistent delayed equations for jump-Markov processes on an adaptive network in our earlier work \cite{maclaurin2024hydrodynamic}. To our knowledge, this is the first time that such a result has been proved in the literature. An analogous phenomenon is
present in spin-glass dynamics (also high-dimensional systems with disordered
connectivity): an autonomous $n$-dimensional particle equation leads to a
non-autonomous limit as $n\to\infty$
\cite{BenArous1995,BenArous2006,Crisanti2018,Helias2020,MacLaurin2024}.

The non-autonomous nature of the limiting dynamics typically makes further analysis
much more difficult. The reason is that most dynamical systems techniques for
studying phenomena such as bifurcations, phase transitions, fixed points and their
stability, Lyapunov Exponents and occupation measures are specifically developed for
autonomous systems \cite{D.Meiss2017}. The focus of this paper is to determine
general conditions under which (i) edges are adaptive and stochastic and (ii) one
obtains an \textit{autonomous} limit for the population density. The autonomous limit is desirable because it renders the analysis of coherent phenomena much more tractable (we demonstrate this with an extended neuroscientific application in Section \ref{Section Applications}). In brief, we assume
a digraph structure, meaning that any edge $j \mapsto k$ is distinct from the edge
$k\mapsto j$, and the dynamics on the $k \mapsto j$ edge is only affected by the
state of the $k^{th}$ node. Roughly-speaking, this model type corresponds to
systems where there are delays, noise, or dynamics whenever one node sends a signal
to another node, but the state of the destination node does not affect the
transmission dynamics. 

%For concreteness, let us quickly discuss which types of neuroscientific model are
%relevant to this paper. The results of this paper would not be accurate for
%\textit{synaptic learning} models where the learning-rate depends on the
%post-synaptic activity. However this paper would be accurate for many models of
%synaptic depression such as  %\cite{rosenbaum2012short}, in which the stochastic
%synapse only has a finite level of resources that can deplete due to high levels of
%spiking. This paper is also accurate for scenarios in which there is a transmission
%delay \cite{Atay2006} when one neuron propagates a signal to another.

%To study the coherent activity surveyed in the above paragraph, it is highly desirable that one finds an autonomous field equation in the large size limit. It is generally much harder to perform tasks such as bifurcation analyses on delay equations. For example questions such as `Is there a fixed point?', 'What is its stability?' are much harder to address in non-autonomous systems. 
%In the analysis of high-dimensional
%complex networks, one has to strike a balance between physical accuracy and analytic tractability.
%It is well-known that neural networks with delays can be particularly prone to periodic behaviors \cite{}.
We apply these equations to a high-dimensional model of the visual cortex in the
brain. There is a long history of scholarship on neural field equations, as surveyed
in the textbooks \cite{Coombes2014} and the review articles
\cite{Bressloff2012,Cook2022}. In various ways many of these models take into account
the fact that the effect of neuron $k$ on neuron $j$ is not instantaneous but
delayed. The delays are due to two main factors: delays in synaptic processing
\cite{matveev2000implications,Rosenbaum2013,Mcdonnell2017,coombes2023neurodynamics}, and delays due to finite propagation
speeds down the dendrites and axons \cite{Atay2006,Bressloff2012}. 

One of the earliest works on neural fields with delays is that of Nunez and
Srinivasan in the 1970's \cite{Nunez1974,nunez1996neocortical}. They postulated that
brain signals are inherently wave-like, and they proposed various field models (i.e.
macroscopic continuum approximations) to take this into account. It is argued that
the incorporation of synaptic delays provides an elegant explanation for the various
peaks in the Fourier Spectrum of electrical recordings of brain activity (i.e. the
synaptic delays induce a resonance at particular frequencies). In the work of Liley
and Wright \cite{Wright1996}, it is argued that long range cortico-cortical fibre
systems lead to significant propagations delays, which must be incorporated into the
model. They employ the telegraphers equation. A similar model was employed by Jirsa
and Haken \cite{Jirsa1996} and Robinson, Rennie and Wright \cite{Robinson1997,mukta2017theory}. Atay
and Hutt found that the incorporation of delays into neural field equations can
excite many patterns \cite{Atay2006}. In other circumstances, delayed feedback can
destroy patterns \cite{hutt2013distributed}. Faye and Touboul found pulsatile solutions in neural fields with delays \cite{faye2014pulsatile}. More recent work has demonstrated how
adaptation and phase-lag can enhance synchronization \cite{martens2016chimera}.
Some recent papers by Folias have discovered a variety of patterns exhibitted by neural fields with
delayed interactions and excitatory / inhibitory balance
\cite{folias2026spatially,folias2026spatially2}. A recent preprint by Parks and
Kilpatrick found stroboscopic motion in neural fields with delays
\cite{parks2026stroboscopic}, and another recent preprint by Omel'chenko and Laing studied the patterns generated
by delays in neural field equations of second generation
\cite{omelchenko2026dynamicsolutionsgenerationneural}. For an overview of the theory of neural fields see \cite{Bressloff2012,Coombes2014,Cook2022}.

The prevalent approach to the modelling of transmission delays in neuroscience is to utilize Delay-Differential-Equations \cite{Ermentrout2010,coombes2023neurodynamics}. Stochastic models of spiking neurons (aka Hawkes Processes) also typically utilize a delay kernel \cite{galves2016modeling,Locherbach2018,Chevallier2017}. In this paper, on the other hand, we seek models
that implement delay and attenuation of a signal in a Markovian setup, through the introduction of auxiliary processes at the microscopic level. In the hydrodynamic limit, we obtain autonomous (non-delayed) Integro-Differential Equations that describe the average neural activity. Retaining a Markovian representation greatly facilitates the toolbox of Dynamical
Systems theory, including (i) the search for fixed points, (ii) stability analysis
and (iii) bifurcation analysis.

In our specific application in Section \ref{Section Applications}, the delay is implemented by extending the state space of the
network, in particular its edge-to-edge dynamics, by an integer factor $l$, where $l$
is a number of layers in a so-called \textit{delay chain}. This type of
representation is sometimes referred to as the \textit{linear  chain trick} (see for
instance \cite{Mocek2005, Hurtado2019, Hurtado2020,
henriknevermannMappingDynamicalSystems2023} and \cite[Section
7.1.2]{smithIntroductionDelayDifferential2011}). In the context of mathematical
neuroscience, a similar delay chain to the one we present here is found in the
numerical approximation of limit cycles in networks with a single
delay~\cite{nicksPhaseAmplitudeResponses2024}.

It is important to highlight that, in this work, we do not seek to approximate a set
of Delayed Differential Equations. We rather consider the model with linear delay
chain as a model in its own right, describing time shift \textit{and
attenuation} of signals by means of a chain of linear filters. The mathematical
ingredients for obtaining such behaviour are already in use in the mathematical
neuroscience community, when modelling synaptic responses as temporal filters \cite{coombes2023neurodynamics}.

%\da[inline]{The next paragraph is dangling}
%\da[inline]{Do we miss a paragraph on the structure of the paper?}
%\da[inline]{At the moment the intro talks a lot about delays, with respect to the
%mathematical mean-field derivation. This seems odd as we probably want to sell more
%the other part of the paper, and put less emphasis on delays?}
There exist various works which study the mean-field limit of related models with dynamic edges. Greven, Den Hollander, Klimovsky and Winter outline a general formalism for studying dynamic graphs \cite{greven2025continuum}. Ganguly, Spiliopoulos and Sussman study a network with nonlinear edge dynamics \cite{ganguly2026mean}. In another work \cite{maclaurin2026large}, we determined the population density limit in the case that the connectivity is extremely sparse (the typical number of connections to any neuron is $O(1)$), and the edge dynamics is nonlinear and deterministic.
More broadly, there has recently been considerable work on pattern formation in particle systems 
\cite{Carrillo2020,Bramburger2023,avitabile2026neural,maclaurin2024hydrodynamic,painter2024biological}. Typically one obtains the same mean-field limit for large random graphs, as long as the density of connections is not too sparse \cite{Chevallier2019,Lucon2020a,AgatheNerine2022}.

The paper is structured as follows. In Section \ref{Section Model Outline}, we outline the model and our assumptions. In Section \ref{Section Main Results}, we outline the main results. In Section \ref{Section Applications}, we apply our results to a model of the visual cortex in the mammalian brain. We choose parameters that lead to a Gaussian limit, and we outline specific `neural field' equations for the mean and variance. We also perform some convergence studies. In Section \ref{Section Proof Outline}, we provide an overview of the key steps in the Proofs. In Section \ref{Section Proof Details}, we provide the proof details. In Section 6.1, we prove that the self-consistent Mckean-Vlasov limit is well-posed. In Section 6.2, we control the probability of rare fluctuations in individual neurons and synapses. In Section 6.3, we prove that an approximate system with frozen interaction converges to the Mckean-Vlasov limit. In Section 6.4, we employ a Coupling argument to prove that the distance between the frozen system and original system goes to zero. In the Appendix, we offer some additional details on how to ensure a Gaussian limit.

%However it is also noted by the author \cite{MacLaurinAdaptive2024} that one can obtain an autonomous hydrodynamic limit, as long as the evolution of the synaptic connection is independent of the postsynaptic activity (and only depends on the presynaptic activity). One of the main aims of this paper is to extend the results in \cite{MacLaurinAdaptive2024} to the continuum case.
%Despite the many many applications, there does not seem to exist a general `Mckean-Vlasov' type equation that yields the large $n$ limiting dynamics. Previous work, including by this author, only obtained limiting equations by resorting to implicit delay-stochastic differential equations \cite{MacLaurin2024_Sparse}. There are several studies of related systems, including \cite{Gkogkas2023}, Many scholars have also considered the hydrodynamic limit of large networks of interacting neurons on inhomogeneous graphs, including \cite{Chevallier2019,Lucon2020a,AgatheNerine2022,Bramburger2023,Bramburger2024,Avitabile2024_2}.
 
 %A very important application of these results is that it yields a general formalism for generating random graphs, complementing for instance \cite{Lovasz2012}. If, for example, the empirical measure concentrates at a single value in the large $n$ limit, then it will automatically yield an accurate understanding of the local structure of the graph. One must also note that there exists a variety of adaptive network models which are not of the mean-field type, such as the adaptive voter model formulated by Durrett  \cite{Basu2017}.

\subsection{Notation}
%We let $\mathbb{M}_q$ denote the set of all $q\times  q$ symmetric matrices that are positive-semi-definite. We endow $\mathbb{M}_q$ with the Euclidean topology inherited from $\mathbb{R}^4$ (it is a closed subset of $\mathbb{R}^4$). 
The index set of the nodes is $\mathbb{N}_n := \lbrace 1,2,\ldots,n \rbrace$. For any
event $\mathcal{A}$, $\mathbf{1}\big\lbrace \mathcal{A} \big\rbrace$ is the indicator
function (i.e. it equals $1$ if the event holds, and $0$ otherwise). For the space
$\mathcal{X}$ with metric $d$, define the Wasserstein Distance  
\[
d_W: \mathcal{P}\big(\mathcal{E} \times \mathcal{E} \times \mathcal{C}( [0,T], \mathbb{R}^d) \times \mathcal{C}( [0,T], \mathbb{R}^c) \big) \times \mathcal{P}\big(\mathcal{E} \times \mathcal{E} \times \mathcal{C}( [0,T], \mathbb{R}^d) \times \mathcal{C}( [0,T], \mathbb{R}^c) \big) \mapsto \mathbb{R}^+
\]
to be
\begin{align}
d_W(\mu,\nu) = \inf_{\zeta} \mathbb{E}^{\zeta}\big[ d\big( (\theta,\eta,z,u) , (\tilde{\theta}, \tilde{\eta} , \tilde{z}, \tilde{u}) \big) \big]
\end{align}
where the law of $(\theta,\eta,z,u) $ is $\mu$, the law of $ (\tilde{\theta}, \tilde{\eta} , \tilde{z}, \tilde{u})$ is $\nu$, and the infimum is over all possible couplings $\zeta$. Here 
\[
d\big( (\theta,\eta,z,u) , (\tilde{\theta}, \tilde{\eta} , \tilde{z}, \tilde{u}) \big) = \min\big\lbrace 1 , d_{\mathcal{E}}(\theta,\tilde{\theta}) + (\eta,\tilde{\eta}) , \norm{z - \tilde{z}}_T + \norm{ u -\tilde{u}}_T \big\rbrace .
\]
and
\[
\norm{z- \tilde{z}}_T = \sup_{t\leq T}\norm{z_t - \tilde{z}_t}.
\]
Throughout this paper, we work in a filtered probability space $\big( \Omega, \mathcal{F} , (\mathcal{F}_t) , \mathbb{P}\big)$ satisfying the usual conditions.
%Let $\Gamma$ and $\Gamma_E$ be discrete sets, specifying the possible states of the node variables and edge variables. Let $\mathcal{D}([0,T], \Gamma)$ and $\mathcal{D}([0,T], \Gamma_E)$ specify the set of all cadlag trajectories taking values in (respectively) $\Gamma$ and $\Gamma_E$. This means that any $x \in \mathcal{D}([0,T], \Gamma)$ must be (i) piecewise constant, (ii) with only a finite number of discontinuities, (iii) possessing left limits and continuous from the right.
\section{Model Outline} \label{Section Model Outline}

We consider the following system of interacting neurons. The neurons are indexed by $\mathbb{N}_n := \lbrace 1,2,\ldots,n \rbrace$ and are embedded in a compact Riemannian Manifold $\mathcal{E}$. Neuron $j$ is given a position $x^j_n$ (considered constant throughout this document). It is assumed that there exists a bounded complete metric $d_{\mathcal{E}}(\cdot,\cdot)$ on $\mathcal{E}$.

We first require the weak convergence of the distribution of neuronal positions.
\begin{hypothesis} \label{Hypothesis Distribution of Positions}
It is assumed that there exists a probability measure $\mu_{\mathcal{E}} \in \mathcal{P}(\mathcal{E})$ such that
\begin{equation}
\lim_{n\to\infty} n^{-1}\sum_{j\in \mathbb{N}_n} \delta_{x^j_n} = \mu_{\mathcal{E}}.
\end{equation}
\end{hypothesis} %\cite{Avitabile2026}
The neurons are connected by a marked digraph. Edge $k \mapsto j$ is assigned a `type
variable' $K^{jk}$.  $K^{jk}$ takes values in a compact subset $\Gamma \subset
\mathbb{R}^{c \times c}$. For example, in one dimension $K^{jk} = 1$ could indicate
the presence of an excitatory edge, and $K^{jk} = -1$ could indicate the presence of
an inhibitory edge and $K^{jk} = 0$ means that there is no edge.  Let us emphasize
that we do not require that $K^{jk} = K^{kj}$. However we do require the following
`graphon'-type assumption on the connectivity, to ensure that a continuum
representation is possible in the large size limit \cite{delattre2016note,lacker2023label}. This assumption parallels that in
our recent paper without delayed interactions \cite{avitabile2026neural}. A similar set of assumptions was originally employed in the work of Lucon on particle systems on inhomogeneous random graphs \cite{Lucon2020a}.
\begin{hypothesis} \label{Hypothesis Graphon}
(i) There exists a continuous function $\mathcal{K} : \mathcal{E} \times \mathcal{E} \mapsto \mathbb{R}^{c \times c}$ such that
\begin{align}
 \lim_{n\mapsto \infty} \delta_n =& 0 \text{ where } \\
 \delta_n =& n^{-1}  \sum_{j\in \mathbb{N}_n} Q^j_n
\end{align}
Here
\begin{equation}
Q^j_n = \sup_{\alpha \in [-1,1]^c}   \big\| R^j_{n} (\alpha) \big\|^2
\end{equation}
and $R^j_{n} : [-1,1]^{c} \mapsto \mathbb{R}^c$ is such that
\begin{align}
R^j_{n} (\alpha) &= n^{-1} \sum_{k\in I_n} \big( \varphi_n^{-1} K_n^{jk} - \mathcal{K}(x^j_n , x^k_n) \big) \alpha^k . %\\R^j_{-,n} (\alpha) &= n^{-1} \sum_{k\in I_n} \big( \varphi_n^{-1} \mathbf{1} \lbrace  K^{jk} = -1 \rbrace - \mathcal{K}_-(x^j_n , x^k_n) \big) \alpha^k .
\end{align}
(ii) There exists a constant $C > 0$ such that for all $y \in \mathbb{R}^{cn}$
\begin{align}
    n^{-1}\varphi_n^{-1}\sum_{j,k \in \mathbb{N}_n} (y^j)^T K^{jk}_n y^k \leq C \sum_{j \in \mathbb{N}_n} \| y^j \|^2.
\end{align}
\end{hypothesis}
\begin{remark}
One way to satisfy Hypothesis \ref{Hypothesis Graphon} is as follows. For $\eta,\theta \in \mathcal{E}$, let $p_{\eta \theta} \in \mathcal{P}( \Gamma )$ be a probability density, and assume that 
\[
(\eta,\theta) \mapsto p_{\eta\theta} \text{ is continuous.}
\]
Let $\lbrace K^{jk}_n \rbrace_{j,k \in I_n}$ be probabilistically independent, and such that $K_n^{jk}$ has probability law $p_{x^j_n , x^k_n}$. A straightforward adaptation of the proof in the Supplementary Materials of \cite{avitabile2026neural} would imply that the requirements of Hypothesis \ref{Hypothesis Graphon} are satisfied with unit probability, and one would have
\[
\mathcal{K}(\theta,\eta) = \mathbb{E}^{p_{\theta\eta}}[ K ].
\]
\end{remark}

 The state dynamics for neuron $j$ is indicated by a vector $u^j_t \in \mathbb{R}^c$, and the dynamics of the $k \mapsto j$ edge is indicated by a vector $z^{jk}_t \in \mathbb{R}^d$, for some positive integer $d$. The dynamics of the nodes is a `mean-field' function of the edges, i.e.
\begin{align}
du^j_t = \bigg( f(u^j_t) +I(t, x^j_n) +  n^{-1}\sum_{k\in \mathbb{N}_n}  \varphi_n^{-1} K_n^{jk}  F (x^j_n , x^k_n, u^j_t, z^{jk}_t) \bigg) dt + \sigma(x^j_n, u^j_t )dW^j_t,
\end{align}
where $\lbrace W^j_t \rbrace_{j\in I_n}$ are independent $\mathbb{R}^c-$valued Brownian Motions, $\sigma: \mathbb{R}^c \mapsto \mathbb{R}^{c\times c}$ is locally Lipschitz, $f: \mathbb{R}^c \mapsto \mathbb{R}^c$ is locally Lipschitz, and $F: \mathcal{E} \times \mathcal{E} \times \mathbb{R}^c \times \mathbb{R}^d \mapsto \mathbb{R}^c$ are locally Lipschitz. We assume for convenience that $F$ is bounded.

The dynamics of the $k\mapsto j$ edge is a function of the state variables at node
$k$, i.e.
\begin{align}
dz^{jk}_t =  G  (x^j_n, x^k_n , z^{jk}_t  ,u^k_t  )  dt + \gamma (x^j_n , x^k_n ,   z^{jk}_t,u^k_t     ) dW^{jk}_t
\end{align}
where $G : \mathcal{E} \times \mathcal{E} \times \mathbb{R}^d  \times \mathbb{R}^c \mapsto \mathbb{R}^c$ is locally Lipschitz and $\gamma : \mathcal{E} \times \mathcal{E}  \times \mathbb{R}^d \times \mathbb{R}^c \mapsto \mathbb{R}^{c \times c}$ is locally Lipschitz. For mathematical convenience, we define the edge dynamics $z^{jk}_t$ even if $K^{jk}_n = 0$ (this is not a problem, because if $K^{jk}_n = 0$ then $z^{jk}_t$ has no effect on the rest of the network).    

 It is widely conjectured that the difference in timescale of inhibitory and
 excitatory connections is a key reason why the brain supports particular resonances
 and rhythms \cite{ buzsaki2006rhythms_2,Brunel2000,Wang2010}. One of the most
 important reasons that we insist that the advection and diffusion terms depend on
 the positions $x^j_n$ and $x^k_n$ is to enable a means of modelling synaptic
 transmission delays, where the delay is a function of the distance $d_{\mathcal{E}}(x^j_n,x^k_n)$,
 amongst other variables. For more details, see our extended example in Section 4.

%\begin{remark}
%Let us underscore the important fact that we are assuming that $dz^{jk}_t$ does not depend on $u^j_t$. This is to ensure that we obtain an autonomous limit for the population density equations. If we did not do this, the limiting equations would be the probability law of a delayed SDE.
%\end{remark}

To prevent the system from blowing up, we assume that the intrinsic neuronal dynamics, and synaptic dynamics, each satisfy a one-sided Lipschitz condition. The fact that we do not require a uniform Lipschitz condition is an important generalization of our earlier work \cite{avitabile2026neural}. Thus our results will hold if (for instance) $f$ is Fitzhugh-Nagumo dynamics \cite{coombes2023neurodynamics}.
\begin{hypothesis} \label{Hypothesis Lipschitz Functions}
There exists a constant $C > 0$ such that for all $u,v \in \mathbb{R}^c$ and all $z,w \in \mathbb{R}^d$, %for all $n \geq 1$, all $\lbrace z^j \rbrace \subset \mathbb{R}^{dn}$ and all $\lbrace J^{jk} \rbrace \subset \mathbb{R}^{c n^2}$, 
\begin{align*}
  \big\langle u - v,  f(u) - f(v) \big\rangle &\leq C \norm{u-v}^2  \\
  \big\langle z - w , G(x,y ,z ,u) - G(x',y' ,w ,v) \big\rangle &\leq C  \norm{z-w}^2 \nonumber \\ &+ C\norm{z-w}\big\lbrace d_{\mathcal{E}}(x,x') +d_{\mathcal{E}}(y,y') + \norm{u-v} \big\rbrace .
  \end{align*}
  \begin{itemize}
\item  The function $F: \mathcal{E} \times \mathcal{E} \times \mathbb{R}^c \times  \mathbb{R}^d$ is globally Lipschitz in its arguments, and uniformly bounded.
\item  The function $f: \mathbb{R}^c \mapsto \mathbb{R}^c$ is uniformly bounded.
\item The functions $\sigma: \mathbb{R}^c \mapsto \mathbb{R}^{c\times c}$ and $\gamma: \mathcal{E}\times\mathcal{E} \times \mathbb{R}^d \times \mathbb{R}^c$ are globally Lipschitz and bounded.
%t\item The function $G: \mathcal{E} \times \mathcal{E} \times \mathbb{R}^c  \times \mathbb{R}^d$ is globally Lipschitz in its arguments, and uniformly bounded.
\end{itemize}
 %+ n^{-1}\sum_{k\in I_n} K^{jk} F( x^j_n , x^k_n , J^{jk} ) \bigg\rangle  + n^{-2} \sum_{j,k\in I_n} \bigg\langle J^{jk} , G_{\alpha}( x^j_n, x^k_n ,z^k ,  J^{jk}) \bigg\rangle \\ \leq
%\frac{C}{n}\sum_{j\in I_n} \big\langle z^j , z^j \big\rangle +  \frac{C}{n^2} \sum_{j,k\in I_n} \big\langle J^{jk} , J^{jk} \big\rangle   + C .
%\end{multline}
%We also assume that
%\begin{align}
%n^{-2}\sum_{j,k\in I_n} \norm{ \gamma( x^j_n , x^k_n , u^k , z^{jk})^T J^{jk} }^2 &\leq \frac{C}{n^2} \sum_{j,k \in I_n} \norm{z^{jk}}^2  \\
%n^{-1}\sum_{j\in I_n} \norm{ \sigma(x^j_n, u^j )^T u^j}^2  &\leq  \frac{C}{n} \sum_{j \in I_n} \norm{u^{j}}^2
%\end{align}
\end{hypothesis}

Define the empirical measure at time $t$ to be
\begin{align}
\hat{\mu}^n_{t} &=  n^{-2} \sum_{j,k\in I_n  } \delta_{x^j_n, x^k_n,   z^{jk}_t , u^k_t } \in \mathcal{P}\big( \mathcal{E} \times \mathcal{E}   \times \mathbb{R}^d \times \mathbb{R}^c \big) .
\end{align}
%\hat{\mu}^n_{t} &= ( n)^{-2}\sum_{j,k\in I_n } \delta_{x^j_n, x^k_n,  ,  J^{jk}_t , z^k_t } \in \mathcal{P}\big( \mathcal{E} \times \mathcal{E}  \times \mathbb{R}^d \times \mathbb{R}^c \big) \\
% \text{ where }\\
%\Xi^n_+ &=  \sum_{j,k\in I_n : K^{jk} = 1}  \\
%\Xi^n_- &=  \sum_{j,k\in I_n : K^{jk} = -1} 
The initial conditions $\lbrace u^j_0 \rbrace_{j\in I_n}$ and $\lbrace z^{jk}_0 \rbrace_{j,k \in I_n}$ are constants.

\begin{hypothesis} \label{Hypothesis Initial Condition Convergence}
It is assumed that there exists $\nu_0 \in  \mathcal{P}\big( \mathcal{E} \times \mathcal{E}   \times \mathbb{R}^d \times \mathbb{R}^c \big)$ such that the following limit exists
\begin{align}
\lim_{n\to\infty} \hat{\mu}^n_{0} &= \nu_{0} .%\\\lim_{n\to\infty} \hat{\mu}^n_{-,0} &= \nu_{0}^-.
\end{align} 
Let $q: \mathcal{E} \times \mathcal{E}  \mapsto \mathcal{P}\big(\mathbb{R}^d \times \mathbb{R}^c \big)$ be a regular conditional probability distribution, i.e. it is such that for measurable sets $A_1,A_2 \subset \mathcal{E}$,
\begin{align}
 \nu_{0}\big(A_1 \times A_2 \times B_1 \times B_2 \big) &= \int_{A_1} \int_{A_2} \int_{B_2} q_{\eta\theta }(B_1 \times B_2  )   \mu_{\mathcal{E}}(d\theta)  \mu_{\mathcal{E}}(d\eta)  .
\end{align}
The initial conditions are assumed to satisfy the uniform moment bound
\begin{align}
\sup_{n\geq 1} \bigg\lbrace n^{-1}\sum_{j\in \mathbb{N}_n} \big\| u^j_0 \big\|^2 \bigg\rbrace < \infty \\
\sup_{n\geq 1}\bigg\lbrace  n^{-2}\sum_{j,k\in \mathbb{N}_n} \big\| z^{jk}_0\big\|^2 \bigg\rbrace < \infty .
\end{align}
%  \nu^-_{0}\big(A_1 \times A_2 \times B_1 \times B_2 \big) &= \int_{A_1} \int_{A_2} \int_{B_2} q^-_{\eta\theta ,z}(B_1  ) d\tilde{q}_{\eta\theta}(z) \mu_{\mathcal{E}}(d\theta)  \mu_{\mathcal{E}}(d\eta) .
\end{hypothesis}
\section{Main Results} \label{Section Main Results}
We write the pathwise empirical measure generated by the system as 
\begin{align}
\hat{\mu}^n &= n^{-2} \sum_{j,k\in I_n } \delta_{x^j_n, x^k_n,   z^{jk}_{[0,T]} , u^k_{[0,T]} } \in \mathcal{Y} %\\
\end{align}
where %\hat{\mu}^n_- &=( \Xi^n_-)^{-1} \sum_{j,k\in I_n: K^{jk} = -1} \delta_{x^j_n, x^k_n,  J^{jk} , z^k } \in \mathcal{Y} 
\begin{align}
\mathcal{Y} =  \mathcal{P}\bigg(  \mathcal{E} \times \mathcal{E}   \times  \mathcal{C}\big( [0,T],   \mathbb{R}^d \big) \times \mathcal{C}\big( [0,T] ,  \mathbb{R}^c \big) \bigg).
\end{align}
Our first main result concerns the almost-sure convergence of the empirical measure. We first specify the law of the limiting probability measure.

For any $\eta,\theta \in \mathcal{E}$, define the stochastic processes $z_{\eta\theta,t} \in \mathbb{R}^d$, $u_{\eta\theta,t} \in \mathbb{R}^c$, such that
\begin{align}
dz_{\eta\theta,t} =& G\big(\eta,\theta,z_{\eta\theta,t}, u_{\eta\theta,t}  \big) dt + \gamma( \eta,\theta,z_{\eta\theta,t},u_{\eta\theta,t} ) dW_{\eta\theta,t} \label{eq: J plus limit} \\
du_{\eta\theta,t} =& \big( f(u_{\eta\theta,t}) + I(\theta,t) +  H (\theta, u_{\eta\theta} ,\mu_{\theta,t}) \big) dt + \sigma( \theta ,  u_{\eta\theta,t} )dW_{\theta,t}, \label{eq: zlimit}
\end{align}
where $\lbrace W_{\eta\theta,t}  , W_{\theta,t} \rbrace_{\eta,\theta \in \mathcal{E}}$ are independent adapted Brownian Motions taking values in (respectively) $\mathbb{R}^d$ and $\mathbb{R}^c$, $\mu_{\eta,t} \in \mathcal{P}\big( \mathcal{E}  \times \mathbb{R}^d \times \mathbb{R}^c \big)$ is such that for any measurable sets $A_1,A_2 \subseteq \mathcal{E}$, and any measurable $B_1 \subseteq \mathbb{R}^d$ and $B_2 \subseteq \mathbb{R}^c$,
\begin{align}
\mu_{\eta,t}\big( A_2 \times B_1 \times B_2 \big) = \int_{A_2} \mathbb{P}\big( z_{\eta\theta,t} \in B_1 , u_{\eta\theta,t} \in B_2 \big)  \mu_{\mathcal{E}}(d\theta),
\end{align}
and
\begin{align}
H:& \mathcal{E} \times \mathbb{R}^c \times \mathcal{P}\big(   \mathcal{E} \times \mathbb{R}^d \times \mathbb{R}^c \big)  \mapsto \mathbb{R}^d \\
 H(\theta,v,\mu) =&  \mathbb{E}^{( \beta,z,u) \sim \mu }\big[    \mathcal{K}(\theta,\beta)F(\theta,\beta,v, z  ) \big].   \label{eq: H definition}
 \end{align} 
The initial conditions  $\big(z_{\eta\theta,0} ,  u_{\eta\theta,0} \big) $ are distributed according to $q_{\eta\theta} \in \mathcal{P}\big( \mathbb{R}^d \times \mathbb{R}^c \big)$. We assume that  $\big(z_{\eta\theta,0} ,  u_{\eta\theta,0} \big) $ is independent of $\big(z_{\alpha\beta,0} ,  u_{\alpha\beta,0} \big) $ if either $\alpha\neq \eta$ or $\beta \neq \theta$. Define the probability law $\mu \in \mathcal{P}\big( \mathcal{E} \times \mathcal{E} \times \mathcal{C}([0,T],\mathbb{R}^d) \times \mathcal{C}([0,T],\mathbb{R}^c) \big)$ to be such that, for measurable $A_1,A_2 \subseteq \mathcal{E}$, measurable $B_1 \subseteq \mathcal{C}([0,T],\mathbb{R}^d)$ and measurable $B_2\subseteq \mathcal{C}([0,T],\mathbb{R}^c) $, it holds that
\begin{align}
\mu\big(A_1\times A_2 \times B_1 \times B_2 \big) = \int_{A_1} \int_{A_2} \mathbb{P}\big( z_{\eta\theta} \in A_1 \; , \; u_{\eta\theta} \in A_2 \big) \mu_{\mathcal{E}}(d\theta)\mu_{\mathcal{E}}(d\eta).
\end{align}
We can now outline our main theorem.
\begin{theorem}\label{Theorem main}
There exist unique continuous stochastic processes that solve the system \eqref{eq: J plus limit} - \eqref{eq: H definition}.  Furthermore, with unit probability,
\begin{align}
\lim_{n\to\infty} d_W\big( \hat{\mu}^n , \mu \big) = 0.
\end{align}
%Furthermore for any measurable sets $A_1,A_2 \subseteq \mathcal{E}$, $B_1 \subset \mathcal{C}\big( [0,T], \mathbb{R}^d \big)$ and $B_2 \subset \mathcal{C}\big( [0,T], \mathbb{R}^c \big)$, it holds that
%\begin{align}
%\lim_{n\to\infty} \hat{\mu}^n \big( A_1, A_2 , B_1 , B_2 \big) = \int_{A_1} \int_{A_2}  \mathbb{P}\big( z_{\eta\theta} \in B_1 \text{ and } u_{\theta} \in B_2 \big)  \mu_{\mathcal{E}}(d\theta)\mu_{\mathcal{E}}(d\eta) .
%\end{align}
\end{theorem}
Let $\mu_{\eta} \in \mathcal{P}\big( \mathcal{E} \times \mathcal{C}([0,T],\mathbb{R}^d) \times \mathcal{C}([0,T],\mathbb{R}^c) \big)$ be the marginal conditional law of $\mu$ (conditioned on the first variable $\eta$). 
\begin{remark}
Its easy to check that for any $\theta,\alpha\in \mathcal{E}$, the probability law of $u_{\eta\theta,[0,T]}$ is the same as the probability law of $u_{\eta\alpha,[0,T]}$.
\end{remark}
We can alternatively characterize the marginals of the limiting law in terms of a Mckean-Vlasov-type Fokker-Planck PDE. This is noted in the following corollary.
%, and
%$\mu_t \in \mathcal{P}\big( \mathcal{E} \times \mathcal{E} \times \mathbb{R}^c \times \mathbb{R}^d \big)$ to be the marginal of $\mu$ at time $t$. Lets first define the limiting probability law $\mu_t \in \mathcal{P}\big( \mathcal{E} \times \mathcal{E}  \times \mathbb{R}^c \times \mathbb{R}^d \big)$, and $\mu \in \mathcal{C}\big( [0,T] , \mathcal{P}\big( \mathcal{E} \times \mathcal{E} \times \mathbb{R}^c \times \mathbb{R}^d \big) \big)$ is defined to be $(\mu_t)_{t\in [0,T]}$. 
Write $\mu_{\theta,t} \in \mathcal{P}\big(   \mathcal{E} \times \mathbb{R}^d \times \mathbb{R}^c \big)$ to be the marginal of $\mu$ at time $t$ and position $\theta$. $\mu_{\theta,t}$ is such that for any measurable $A  \subseteq \mathcal{E}$, $B_1 \subset \mathbb{R}^d$ and $B_2 \subset \mathbb{R}^c$,
\begin{align} \label{eq: mu specification 1}
\mu_{\theta,t}(A \times   B_1 \times B_2) = \int_{A}  \int_{B_1} \int_{B_2} p_{\theta\eta,t}(z , u)\; dz\; du\;   d\mu_{\mathcal{E}}(\eta),
\end{align}
where for any $\theta,\eta \in \mathcal{E}$, $p_{\theta\eta,t}: \mathbb{R}^d \times \mathbb{R}^c \mapsto \mathbb{R}^+$ solves the partial differential equation
\begin{multline} \label{eq: p specification 1}
\partial_t p_{\theta\eta,t}(z,u) = - \nabla_u \cdot\bigg\lbrace  p_{\theta\eta,t}(z,u) \bigg(f(u) +I(\eta,t)+ H(\eta,u,\mu_{\theta,t})  \bigg)\\ - \frac{1}{2}\nabla_u \bigg( \sigma(\eta,u) \sigma(\eta,u)^T p_{\theta\eta,t}(z,u) \bigg) \bigg\rbrace \\
- \nabla_z \cdot \bigg\lbrace   p_{\theta\eta,t}(z,u)  G(\theta,\eta,u,z) - \frac{1}{2}\nabla_{z} \bigg(\gamma(\theta,\eta,u,z)\gamma(\theta,\eta,u,z)^T p_{\theta\eta,t}(z,u) \bigg) \bigg\rbrace 
\end{multline}
and
\begin{align} \label{eq: mu specification 3}
H(\eta,v,\mu_t)= \int_{\mathcal{E}} \int_{\mathbb{R}^c} \int_{\mathbb{R}^d} \mathcal{K}(\eta,\theta) F(\eta,\theta,v,z) p_{\eta\theta,t}(z,u)du dz d\mu_{\mathcal{E}}(\theta) .
\end{align}
and the initial condition is
\[
p_{\theta\eta,0} = q_{\theta\eta}.
\]
In the above $\nabla_z$ represents the gradient with respect to the $z$ variables
only and $\nabla_u$ represents the gradient with respect to the $u$ variables only. In the Appendix, we explain how one can choose parameters to ensure that $\mu$ is Gaussian.
 
%\begin{theorem}
%There exists a unique $(\mu_t)_{t\in [0,T]} \in \mathcal{C}\big( [0,T] , \mathcal{P}\big( \mathcal{E} \times \mathcal{E} \times \mathbb{R}^c \times \mathbb{R}^d \big) \big)$ satisfying \eqref{eq: mu specification 1}-\eqref{eq: mu specification 3}. Furthermore, $\mathbb{P}$-almost-surely,
%\begin{align}
%\lim_{n\to\infty} \sup_{t\leq T} d_W\big( \hat{\mu}^n_t, \mu_t \big) = 0.
%\end{align}
%\end{theorem}

\section{Application} \label{Section Applications}

In this section we outline an extended application motivated by neuroscience. As was
surveyed in the introduction, it is well-known that neuronal networks with both transmission delays
and spatial structure can elicit a range of interesting coherent patterns and oscillations. Delays in
neural transmissions are due to two main effects: synaptic filtering and axonal transmission delays \cite{rosenbaum2012short,coombes2023neurodynamics}.
% For the synaptic timescales, one must
% typically distinguish between the rise and fall times (to account for the
% transmission delays \cite{Brunel2003,Mcdonnell2017}). %This sort of model is often
% known as a double-exponential model of synaptic timescales. We take these constants
% to be functions of the ring-distance.

\subsection{Implementing the Delay and Attenuation of a Signal}
\label{Subsection Delay and Attenuation of a Signal}
In mathematical neuroscience, delayed signals from one neuron to another are
usually implemented using delay-differential-equations
\cite{Ermentrout2010,coombes2023neurodynamics}. As stated above, we want use the
so-called \textit{linear  chain trick} (see for instance
\cite{Mocek2005, Hurtado2019, Hurtado2020, henriknevermannMappingDynamicalSystems2023} and \cite[Section
7.1.2]{smithIntroductionDelayDifferential2011}) to derive a set of ODEs which model
a delay.
% In this paper we seek for models that
% implement delay and attenuation of a signal in a Markovian setup. As we see
% below, this can be obtained by extending the state space of the network, in
% particular its edge-to-edge dynamics, by an integer factor $l$, where $l$ is a number
% of layers in a so-called \textit{delay chain}.

% Retaining a Markovian representation greatly facilitates the toolbox of Dynamical
% Systems theory, including (i) the search for fixed points, (ii) stability analysis
% and (iii) bifurcation analysis. Furthermore the Markovian representation is
% consistent with the mean-field limit obtained in Theorem \ref{Theorem main}.
% Accordingly, we start by outlining a general theory for representing a delayed and
% attenuated signal in terms of an autonomous ODE. 
% This type of representation is
% sometimes referred to as the 
% In the context of mathematical
% neuroscience, a similar delay chain to the one we present below is found in the
% numerical approximation of limit cycles in networks with a single
% delay,\cite{nicksPhaseAmplitudeResponses2024}.

% Here we take a different perspective, even though we formally arrive at a similar
% system. 
% In what follows, we do not aim to approximate a delay equation, but rather to propose
% an alternative model to a delayed system, which sits on equal standing and is
% tractable. 
Assume $f(t)$ is defined for $t \in \RSet_{\geq 0}$, we want to \textit{approximate} a
process that delays $f$ by
a positive time $\tau$,
\[
  y(t) = f(t-\tau)H(t-\tau), \qquad t \in \RSet_{\geq 0}, \qquad \tau \in \RSet_{> 0}
\] 
where $H$ is the Heaviside function. We aim to replace $y(t)$
by another function of time that: (i) solves an ODE, not a DDE; (ii) produces a
delayed and an attenuated $f(t)$ (this is also a desirable feature of the model);
(iii) it is tunable with some parameters. 

If we denote by $Y(s)$ and $F(s)$ the Laplace Transforms of $y$ and $f$,
respectively, it holds that
\[
  e^{\tau s} Y(s) = F(s), \qquad s \in \CSet.
\]
With the view of replacing $y(t)$, we consider a replacement of $Y(s)$. We consider
the product representation of an exponential, which holds also on complex-valued exponentials,
\[
  e^{\tau s} = \lim_{n \to \infty} \Bigl(1 + \frac{\tau s}{n} \Bigr)^{n}.
  \qquad s \in \CSet.
\]
We now fix $l \in \NSet$, and use $l$ products to approximate the exponential, obtaining 
\[
  \Bigl(1 + \frac{\tau}{l} s\Bigr)^{l}  Y_l(s) =   F(s), \qquad s \in \CSet, 
\]
or equivalently, introducing auxiliary variables $Y_1(s), \dots, Y_{l-1}(s)$
\[
 \Bigl(1 + \frac{\tau}{l} s\Bigr) Y_1(s) = F(s),
 \qquad 
 \Bigl(1 + \frac{\tau}{l} s\Bigr) Y_2(s) = Y_1(s),
 \ldots \qquad 
 \Bigl(1 + \frac{\tau}{l} s\Bigr) Y_l(s) = Y_{l-1}(s).
\]
We expect that the time-dependent function $y(t)$ delays approximately $f(t)$,
and also that it attenuates the signal.% \da[]{I think this claim can be
%substantiated by looking at the integral over time of $y(t)$}

The chain of equalities above in frequency domain have a counterpart in time domain,
as ODEs
\[
  \begin{aligned}
     % \frac{\tau}{l} y'_1(t) & = f(t) - y_1(t) \\
     % \frac{\tau}{l} y'_2(t) & = y_1(t) - y_2(t) \\
     %                        & \dots \\
     % \frac{\tau}{l} y'_l(t) & = y_{l-1}(t) - y_l(t)
     \tau/l\; y'_1(t) & = f(t) - y_1(t) \\
     \tau/l\; y'_2(t) & = y_1(t) - y_2(t) \\
                            & \dots \\
     \tau/l\; y'_l(t) & = y_{l-1}(t) - y_l(t),
  \end{aligned}
\]
that is, we have a chain of ODEs, the first one using $f(t)$, and the last one giving
the desired approximation $y_l(t)$ to $y(t)$, in which each ODE has time constant
$\tau/l$. This suggests a modelling strategy to mimic a delay $\tau$ in $f(t)$ by the
following system of ODEs, for \textit{finite, possibly large} $l$.
\[
  \frac{\tau}{l} z'(t) = -A z(t) + h(f(t)), \qquad  t \in \RSet_{\geq  0}, \qquad z(0) = 0,
\]
with
\[
  A = 
  \begin{bmatrix} 
    1 &        &        &   \\
    -1 & 1     &        &   \\
       & \ddots & \ddots &   \\
       &        &   -1   &1 
  \end{bmatrix},
  \qquad 
  h \colon \RSet^l \to \RSet^l, 
  \qquad 
  h(z) = 
  \begin{bmatrix} z \\ 0 \\ \vdots \\ 0 \end{bmatrix}. 
\]
%We can produce \da[]{I have done it already} numerical evidence of this, and also of how the approximation gets
better as $l$ increases.

\subsection{Specific Neuronal Model}
\label{sec:neuromodel}
%Experimental
%studies reveal not a single axonal speed but statistically distributed speeds in cortico-cortical connections mammals. In all studies, the histogram of the axonal speeds follows a gamma distribution with maxima between 5 m/s and 12 m/s in rats and about 0.2 m / s in the brain of cats and monkeys \cite{Atay2006}.  In this section we focus on a relatively simple ring model. See also the work of Roxin, Hansel and Brunel \cite{roxin2005role} for a similar model.  

We now combine the ODEs found above with a specific model of a network of excitatory (E) and inhibitory (I)
neurons. It is widely believe that the interplay of excitatory and inhibitory neurons is responsible for a rich range of phenomena in the brain \cite{rosenbaum2014balanced,maclaurin2024hydrodynamic,maclaurin2025kinetic}. Each neuron is assigned a position $x^j_n$ in a compact domain $\mathcal{E}
= \mathbb{T}^1$ (the torus is used for the famous Ring Model of orientation
selectivity in the visual cortex \cite{Hubel1962,Hubel1962a, Kilpatrick2013a, medathati2017recurrent,MacLaurin2020a}). Each
neuron is assigned a type $\alpha \in \lbrace E,I \rbrace$ and index $j \in
\mathbb{N}_n$. Assume that there are approximately as many excitatory as inhibitory
connections, and that the connectivity is sparse with the average number of inputs to
any neuron scaling as $O\big( n \varphi_n\big)$, where the sparsity factor $\varphi_n
= n^{-q} $ for some $0 < q < 1$. Let $K^{\alpha\beta jk} \in \lbrace -1,0,1\rbrace$
represent the effect of presynaptic neuron $(\beta,k)$ on postsynaptic neuron
$(\alpha,j)$. Since excitatory neurons can only excite other neurons (as long as they
are connected), and inhibitory neurons only inhibit other neurons, we have that for
any $\alpha\in \lbrace E,I \rbrace$,
\begin{align}
K^{\alpha e jk} &\in \lbrace 0,1 \rbrace \label{eq: K constraint 1} \\
K^{\alpha i jk} &\in \lbrace -1,0 \rbrace .\label{eq: K constraint 2}
\end{align}
We sample the connectivity from a random `graphon' type model. We assume that the edge variables $\lbrace K^{\alpha\beta jk} \rbrace$ are mutually independent. Let $p_{\alpha\beta} \in  \mathcal{C}\big( \mathcal{E} \times \mathcal{E} , [0,1] \big)$ be the probability of a connection from neurons of type $\beta$ to neurons of type $\alpha$ (this is taken to be a function of the positions of the neurons). That is,
\begin{align}
\mathbb{P}\big( K^{eejk} = 1 \big) &= \varphi_n p_{ee}\big(x^j_n, x^k_n \big) \\
\mathbb{P}\big( K^{iejk} = 1 \big) &= \varphi_n p_{ie}\big(x^j_n ,  x^k_n  \big)  \\
\mathbb{P}\big( K^{eijk} = -1 \big) &= \varphi_n p_{ei}\big(x^j_n, x^k_n \big) \\
\mathbb{P}\big( K^{iijk} = -1 \big) &= \varphi_n p_{ii}\big(x^j_n, x^k_n \big) 
\end{align}
Note that, thanks to \eqref{eq: K constraint 1}-\eqref{eq: K constraint 2}, it holds that 
\begin{align}
\mathbb{P}\big( K^{\alpha\beta jk} = 0 \big) &= 1 - \mathbb{P}\big( K^{\alpha\beta jk} \neq  0 \big).
\end{align}
%Define the average connectivity
%\begin{align}
%\mathcal{K}(\theta,\alpha) = p_+(\theta,\alpha) - p_-(\theta,\alpha)
%\end{align}

We write $u^{\alpha j}_t \in \mathbb{R}$ to represent the activity of neuron $j$ of type $\alpha$ at time $t$. Motivated by the discussion in Subsection \ref{Subsection Delay and Attenuation of a Signal}, we assume that there are $l \geq 1$ timescales associated with the synaptic delays, and we let $\big( z^{\alpha\beta jk r} \big)_{r=1}^l$ represent the corresponding synaptic variables corresponding to the propagation of a signal from neuron $(\beta,k)$ to neuron $(\alpha,j)$. The heuristic is that $z^{\alpha\beta jk (r+1)} $ is driven by $z^{\alpha\beta jk r} $, and the postsynaptic neuron $(\beta,k)$ is driven by $z^{\alpha\beta jk l}$. 

The dynamics assumes the form
  \begin{align*} 
  du^{\alpha j}_t =&\frac{1}{\tau^\alpha_u(x^j)}
  \biggl[
    -u^{\alpha j}_t + I^\alpha(t, x^j) 
    \begin{aligned}[t]
    &+ \frac{1}{n \varphi_n}
    \sum_{\beta \in \{ E, I \}} \sum_{k \in \NSet_n} K^{\alpha \beta j k} f^{\alpha\beta}(z_t^{\alpha\beta j k l}) \biggr] dt \nonumber \\  
    & + \sigma_u^{\alpha}(t,x^j) dW^{\alpha j}_{u,t} \\  
   \end{aligned}\\
  dz^{\alpha \beta jkr}_t =& \frac{l}{\tau^{\alpha\beta}_z(x^j,x^k)}
  \biggl[
    u_t^{\beta k}
    \delta_{1r} -
    \sum_{s \in \NSet_l} A^{\alpha\beta rs} z_t^{\alpha \beta jks}
  \biggr] dt +
  \sigma_z^{\alpha\beta r}(t,x^j,x^k) dW^{\alpha\beta j k r}_{z,t} 
  \end{align*}
The model features spatially-dependent activity timescales,
and delay functions
\[
  \tau^\alpha_{u} \colon D \to \RSet_{>0}, 
  \qquad 
  \tau^{\alpha\beta}_{z} \colon D \times D \to \RSet_{>0}, 
  \qquad 
  \alpha \in \{ E, I\},
\]
deterministic forcing and stochastic intensity functions
\[
  I^\alpha \colon \RSet_{\geq 0} \times D \to \RSet, 
  \quad 
  \sigma_u^{\alpha \beta} \colon \RSet_{\geq 0} \times D \to \RSet,
  \quad 
  \sigma_z^{\alpha \beta} \colon \RSet_{\geq 0} \times D \times D \to \RSet,
  \quad 
  \alpha \in \{ E, I\},
\]
and in which we have denoted by $A^{rs}$ the components of the matrix
\[
  A = 
  \begin{bmatrix} 
    1 &        &        &   \\
    -1 & 1     &        &   \\
       & \ddots & \ddots &   \\
       &        &   -1   &1 
  \end{bmatrix} \in \RSet^{l \times l}
\]
The Brownian motions are all adapted to the filtration and independent.

%Hence the neuronal dynamics is of the form
%\begin{align}
%dz^j_t =\tau_{neuron}^{-1} \bigg( - z^j_t + n^{-1}\sum_{k\in I_n}  \varphi_n^{-1} K^{jk}  F( J^{jk}_t) \bigg) dt + \sigma W^j_t,
%\end{align}
%where $\lbrace W^j_t \rbrace_{j\in I_n}$ are independent $\mathbb{R}-$valued Brownian Motions, and $F: \mathbb{R} \mapsto \mathbb{R}$ is sigmoidal. See for instance the discussion in \cite{Ermentrout2010} concerning different synaptic models. The model employed here (taking the sigmoidal to be a function of the synaptic current $J^{jk}$ closely resembles some models in \cite{Ermentrout2010}).
% The dynamics of the $k\mapsto j$ edge is a function of the state variables at node $k$, i.e. for some continuous function
%\begin{align}
%\tau_{rise}\in & \mathcal{C}\big( \mathbb{S}^1 , \mathbb{R}_{\geq 0} \big) \\
%\tau_{decay} \in & \mathcal{C}\big( \mathbb{S}^1 , \mathbb{R}_{\geq 0} \big).
%\end{align}
%it holds that
%\begin{align}
%dJ^{jk}_t =& \big( -  \tau_{decay}\big( \norm{ x^j_n - x^k_n }_{\mathbb{S}^1} \big)^{-1} J^{jk}_t +  \tilde{J}^{jk}_t  \big) dt \\
%d\tilde{J}^{jk}_t =&  \tau_{rise}\big( \norm{ x^j_n - x^k_n }_{\mathbb{S}^1} \big)^{-1} \big( -\tilde{J}^{jk}_t + z^k_t \big)dt + \sigma_{syn} dW^{jk}_t.
%\end{align}
% Let us underscore the fact that many scholars might argue that the dominant source of noise in neuronal networks is synaptic transmission failure \cite{manwani2001detecting,maclaurin2024hydrodynamic}.
We define the empirical measure corresponding to presynaptic neurons of type $\beta$ and postsynaptic neurons of type $\alpha$,
\begin{align}
\hat{\mu}^{n,\alpha\beta}_{t} &=\bigg(  \sum_{j,k\in I_n  } \bigg)^{-1} \sum_{j,k\in I_n } \delta_{x^j_n , x^k_n , z^{\alpha\beta jk}_t ,  u^{\beta k}_t}  .
\end{align}
%\hat{\mu}^n_{-,t} &=\bigg(  \sum_{j,k\in I_n : K^{jk} = -1} \bigg)^{-1} \sum_{j,k\in I_n : K^{jk} = -1} \delta_{x^j_n , x^k_n , J^{jk}_t , z^k_t} .

The limit of the empirical measure $\hat{\mu}^{n,\alpha\beta}_{t} $ at time $t$ is a measure $\mu^{\alpha\beta}_t$ on $\mathbb{T}^1 \times \mathbb{T}^1 \times \mathbb{R}^l \times \mathbb{R}$. It is such that
\begin{align}
\mu^{\alpha\beta}_t\big( A \times B \times C \times D \big) =\frac{1}{(4\pi)^2} \int_{A} \int_B \int_{C} \int_D p^{\alpha\beta}_{xy,t}(z,u) du dz dx dy, 
\end{align}
Here $p^{\alpha\beta}_{xy,t}$ is a Gaussian kernel, and can be characterized in terms of its mean and variance, i.e. there are continuous functions
\[
  \begin{aligned}
   m^\alpha_u & \colon [0,T] \times \mathcal{E} \to \RSet, 
              &&\alpha \in \{ E, I \},\\
   m^{\alpha \beta r}_z & \colon [0,T] \times \mathcal{E} \times \mathcal{E} \times \RSet, 
              && \alpha \in \{ E,I \}, \qquad r \in \NSet_l\\
   V^{\alpha\beta}   & \colon [0,T] \times \mathcal{E} \times \mathcal{E} \times \RSet^{(l+1) \times (l+1)},
              && \alpha \in \{ E,I \}
  \end{aligned}
\]
in which $V^{\alpha\beta}$ has components
\[
  V^{\alpha\beta} = 
  \begin{bmatrix} 
    V^{\alpha\beta}_{uu} & V^{\alpha\beta}_{u z_1}  & \ldots & V^\alpha_{u z_l}   \\
    V^{\alpha\beta}_{z_1u}& V^{\alpha\beta}_{z_1 z_1}  & \ldots & V^{\alpha\beta}_{z_1 z_l}   \\
            \vdots &    \vdots           & \ddots &       \vdots   \\
    V^{\alpha\beta}_{z_1 u}& V^{\alpha\beta}_{z_l z_1}  & \ldots & V^{\alpha\beta}_{z_l z_l} 
  \end{bmatrix} 
\]
such that
\begin{align}
m^{\alpha\beta r}_z(t,x,y) =& \int_{\mathbb{R}^l}\int_{\mathbb{R}}z^r p^{\alpha\beta}_{xy,t}(z,u) du dz \\
m^{\alpha}_u(t,x,y) =& \int_{\mathbb{R}^l}\int_{\mathbb{R}}u p^{\alpha\beta}_{xy,t}(z,u) du dz \\
    V^{\alpha\beta}_{uu} =& \int_{\mathbb{R}^l}\int_{\mathbb{R}}
   \big( u - m^{\alpha}_u(t,x,y) \big)^2   p^{\alpha\beta}_{xy,t}(z,u) du dz \\
    V^{\alpha\beta}_{zz} =& \int_{\mathbb{R}^l}\int_{\mathbb{R}}
   \big( z - m^{\alpha\beta}_z(t,x,y) \big)   \big( z - m^{\alpha\beta}_z(t,x,y) \big)^T p^{\alpha\beta}_{xy,t}(z,u) du dz   \\
    V^{\alpha\beta}_{uz} =& \int_{\mathbb{R}^l}\int_{\mathbb{R}}
   \big( u - m^{\alpha\beta}_u(t,x,y) \big)   \big( z - m^{\alpha\beta}_z(t,x,y) \big)^T p^{\alpha\beta}_{xy,t}(z,u) du dz    
\end{align}
These solve the ordinary differential equations
\[
  \begin{aligned}
  & 
  \partial_t m_u^{\alpha}(t,x) = \frac{1}{\tau^\alpha_u(x)}
  \biggl[
    \begin{aligned}[t]
    &-m_u^{\alpha}(t,x) 
    + 
    I^\alpha(t, x) 
      \\
    &
    +\sum_{\beta \in \{ E, I \}} 
    \int_{\mathcal{E}}
    \calK^{\alpha \beta}(x,y)
    S_{\beta}\big(m_z^{\alpha\beta l}(t,x,y) , V^{\alpha\beta(l+1)(l+1)}(t,x,y)\big) dy\bigg]  
    \end{aligned}
    \\  
  &
  \partial_t m_z^{\alpha\beta r}(t,x,y) = \frac{1}{\tau^{\alpha\beta}_z(x,y)}
  \biggl[
    m_u^{\beta}(t,y) \delta_{1r} -
    \sum_{s \in \NSet_l} A^{\alpha\beta rs}
    m_z^{\alpha\beta s}(t,x,y)
  \biggl] \\
  &
  \partial_t V^{\alpha\beta}(t,x,y) = B^{\alpha\beta}(x,y) V^{\alpha\beta}(t,x,y) 
  + V^{\alpha \beta}(t,x,y) B^{\alpha\beta}(x,y)^T + \Lambda^{\alpha\beta}(x,y) \label{eq:variance dynamics}
  \end{aligned}
\]
in which
%\da[]{Could there be an error in sign in the matrix $B^{\alpha \beta}$? Do we change sign along the
%diagonal?}
\[
  \begin{aligned}
 & B^{\alpha \beta}(t,x,y) =
  \begin{bmatrix} 
    -1/\tau^\alpha_u(x)   &                           &                       &   \\
     1/\tau^{\alpha\beta}_z(x,y) &  -1/\tau^{\alpha\beta}_z(x,y)    &                       &   \\
                          &  \ddots                   &  \ddots               &   \\
                          &                           &   1/\tau^{\alpha\beta}_z(x,y)    & -1/\tau^{\alpha\beta}_z(x,y)  
  \end{bmatrix} \in \RSet^{(l+1) \times (l+1)},
  \end{aligned}
\]
and
\[
 \Lambda^{\alpha \beta}(x,y) = \diag\big(  (\sigma_u^\beta)^2 , (\sigma_z^{\alpha\beta 1})^2 , \ldots ,(\sigma_z^{\alpha\beta l})^2 \big)
\]
\[
S^{\alpha\beta}( m ,v ) =  \int_{\mathbb{R}}   f^{\alpha\beta}(j) \rho \big(j , m , v\big) dj  .
\]
% \[
%   \begin{aligned}
%   & \tau_u \partial_t m_u(t,x) 
%     = -L m_u(t,x) + \int_{D} \calK(x,y)F(x,y,m_u(t,y),m_{z_l}(t,y,x),
%     V_{z_lz_l}(t,y))\,dy \\
%   & \frac{\tau_z(x,y)}{l} \partial_t m_z(t,x,y) = A m_z(t,x,y) + h((m_u(t,y))_1)
%   \end{aligned}
% \]
% in which 
% \da[inline]{things to discuss: (i) we should fix notation and stick with it, at the
% moment it is not consistent and I can not navigate the document well, (ii) the
% evolution equation for the variances, written in terms of $A$ (I hope). (iii) which
% entries of the variance matrix make it into the integral.}
% \da[inline]{discuss the option with James of writing this in function space, with a
% function with values in a Banach space}
Note that one can check that the variance equations are consistent. That is, for $\alpha ,\beta,\eta \in \lbrace E,I \rbrace$,
\begin{align}
V^{\alpha\beta 11}(t,x,y) = V^{\eta\beta 11}(t,x,y)
\end{align}
The next Lemma, which is already known (and whose proof we neglect), pertains to the
steady states of the variance equation.
\begin{lemma}
There exists a unique symettric positive-definite matrix $V_*^{\alpha\beta} \in \mathbb{R}^{(l+1) \times (l+1)}$ such that
\begin{align}
B^{\alpha\beta}(x,y) V_*^{\alpha\beta}(x,y) 
  + V_*^{\alpha \beta}(x,y) B^{\alpha\beta}(x,y)^T + \Lambda^{\alpha\beta}(x,y) = 0
\end{align}
\end{lemma}
For a proof, see \cite{Avitabile2024}.

Notice that the evolution of the variance separates from the mean dynamics. In fact, the variance \eqref{eq:variance dynamics} converges to $V_*^{\alpha\beta}$ exponentially. Since this transient dynamics is not our primary interest, for convenience, we assume that the variance is initially in equilibrium, i.e. we assume that for all $x,y\in \mathcal{E}$,
\begin{align}
V^{\alpha\beta}(0,x,y) = V_*^{\alpha\beta}(x,y).
\end{align}
For convenience we employ the following microscopic firing rate function
 \begin{equation}\label{eq:particleFiringRate}
    f(u) = \Phi(\alpha(u-\theta)), \qquad \Phi(u) = \frac{1}{2}\biggl[1 + \erf\biggl(\frac{u}{\sqrt{2}}\biggr)\biggr],
    \qquad 
    \alpha \in \RSet_{>0},
    \quad
    \theta \in \RSet,
  \end{equation}  
  which results in a mean-field firing rate of the form  \cite{Touboul2012}
  \begin{equation}\label{eq:firingRateMeanField}
    F(m,v) = \Phi\biggl(\alpha \frac{m-\theta}{\sqrt{1+\alpha^2 v}}\biggr).
  \end{equation}
\subsection{Results}
\label{ssec:Daniele} 

We performed simulations of the system in \cref{sec:neuromodel}
on a ring $D =
\RSet/(2L \ZSet)$ with $L = 10 \pi$, using a delay chain with $l = 2$, a synaptic kernel
\[
  \mathcal{K}^{\alpha \beta}(x,y) = A^{\alpha \beta} \exp( -B^{\alpha, \beta} |x-y|),
\qquad \alpha, \beta \in \{ E, I \}, \qquad x,y \in D,
\]
null external forcing $I^{\alpha\beta}(x,t) \equiv 0$, firing rate given by
\cref{eq:particleFiringRate}, delay functions
\[
  \tau^\alpha_u(x,y) \equiv \tau^{\alpha}_{u0},
  \qquad 
  \tau^{\alpha \beta}_z(x,y) = \tau_{z0} + \tau_{z1} |x -y|
\qquad \alpha, \beta \in \{ E, I \}, \qquad x,y \in D,
\]
and noise intensities
\[
  \sigma^{\alpha \beta}_{u}(x,y) = \sigma^{\alpha \beta}_z(x,y) \equiv \sigma_0.
\qquad \alpha, \beta \in \{ E, I \}, \qquad x,y \in D,
\]
with parameters reported in \cref{tab:pars}.
Simulations are performed with $n =512$ nodes, which results in $2n + 8n^2 =
2098176$ unknowns. The particle system is evolved using a default Euler-Maruyama
scheme with $dt = 0.005$, and the mean field with a forward Euler scheme with the
same parameters.

 %
% in which $\mathbf{1}^{\alpha, \beta}_{\varepsilon}$ is a thresholding factor equal to $0$ if
% $|A^{\alpha \beta} \exp( -B^{\alpha, \beta} |x-y|)| \leq \varepsilon$

\Cref{fig:oscillatory_pattern} shows that both the particle and the mean field models
support spatio-temporal oscillatory patterns (see also the time snapshot given in
\cref{fig:weak_error_profiles}) for the same parameter values. We stress that a
bifurcation analysis of the model is outside the scope of the present paper, and thus
a study of how oscillations are generated is left for future work. Here we only aim to
give an example of a network in which excitation, inhibition, delay and attenuation
are present in both models and produce oscillations for the same parameter values.

In \cref{fig:weak_error_convergence} we present numerical experiments for the weak
error between the mean-field and the particle model. The results are obtained using
a weak error defined analogously to our previous work \cite[Section 5.3]{avitabile2026neural}, by
integrating against harmonic spatial functions with varying wavelength
% $k=0,1,\ldots,5$, 
and using the mean-field solution at the same time instant shown in
\cref{fig:weak_error_profiles} for the particle system. We note that, in the weak
error analysis, increasing the grid size beyond $n = 512$ is challenging, as the
number of unknowns
in the model scales as $8 n^2$ when $n$ increases. In \cite{avitabile2026neural} we
observed an $O(n^{-1/2})$ scaling when $n$ is in the order of the $2 \cdot 10^3$ and
higher. In the present experiment, those values of $n$ are not accessible as they
will require dedicated parallel numerical schemes, but we see a trend in the right
direction for the values of $n$ we computed.

% and . In
% practice, we were not able to take $n$
% to be very large in the stochastic simulations, because the dimension of the SDE
% scales as $n^2$ rather than $n$.

\begin{table}
  \centering
  \caption{Parameters used in numerical simulations}
  \label{tab:pars}
  \begin{tabular}{cccc}
    \hline
    Parameter & Value & Parameter & Value \\
    \hline\\[0.03em]
    $A^{EE}$ & 3.5 & $B^{EE}$ & 0.15 \\
    $A^{EI}$ & -2.8 & $B^{EI}$ & 0.15 \\
    $A^{IE}$ & 2.5 & $B^{IE}$ & 0.15 \\
    $A^{II}$ & -0.9 & $B^{II}$ & 0.15 \\
    $\tau_{u0}^E$ & 1.0 & $\tau_{u0}^I$ & 4.0 \\
    $\tau_{z0}$ & 0.06 & $\tau_{z1}$ & 0.5 \\
    $\sigma_{0}$ & 0.35 & $l$ & 2\\[0.2em]
    \hline
  \end{tabular}
\end{table}

\begin{figure}
  \centering
  \includegraphics{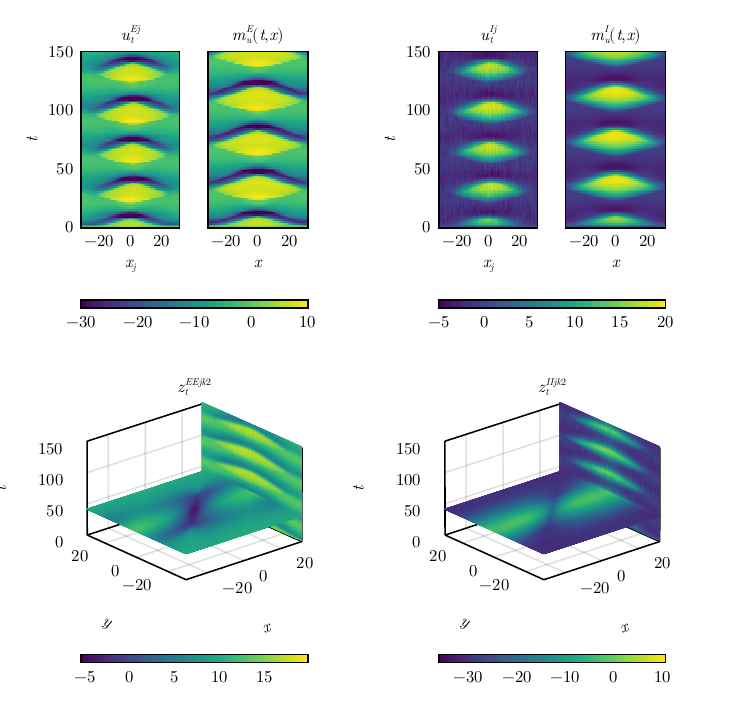}
  \caption{Example of an oscillatory solution. For the variables $u^\alpha_t$, with
  $\alpha \in \{ E, I \}$, we present both particle and mean field solution. We note
that the mean-field solution displays slower oscillations. For the synaptic variables
we choose to show only sections of spatio-temporal profiles of selected variables in
the particle system, namely $z^{EEjk2}_t$ and $z^{IIjk2}_t$. Time dependence of all
variables in the particle system can be seen here:
\href{https://doi.org/10.6084/m9.figshare.31814269}{animation}. Parameters of the
model can be found in \cref{tab:pars} and \cref{ssec:Daniele}.
% \href{https://www.dropbox.com/scl/fi/blp24n3p18sn8tfgktbvw/animation-oscillatory-pattern-particles.mp4?rlkey=2ns0lp72zblvw6tvptafit14q&dl=0}{animation}
}
%\da[inline]{Put animation on a Figshare once finalised, remove it from Dropbox}
  \label{fig:oscillatory_pattern}
\end{figure}
\begin{figure}
  \centering
  \includegraphics{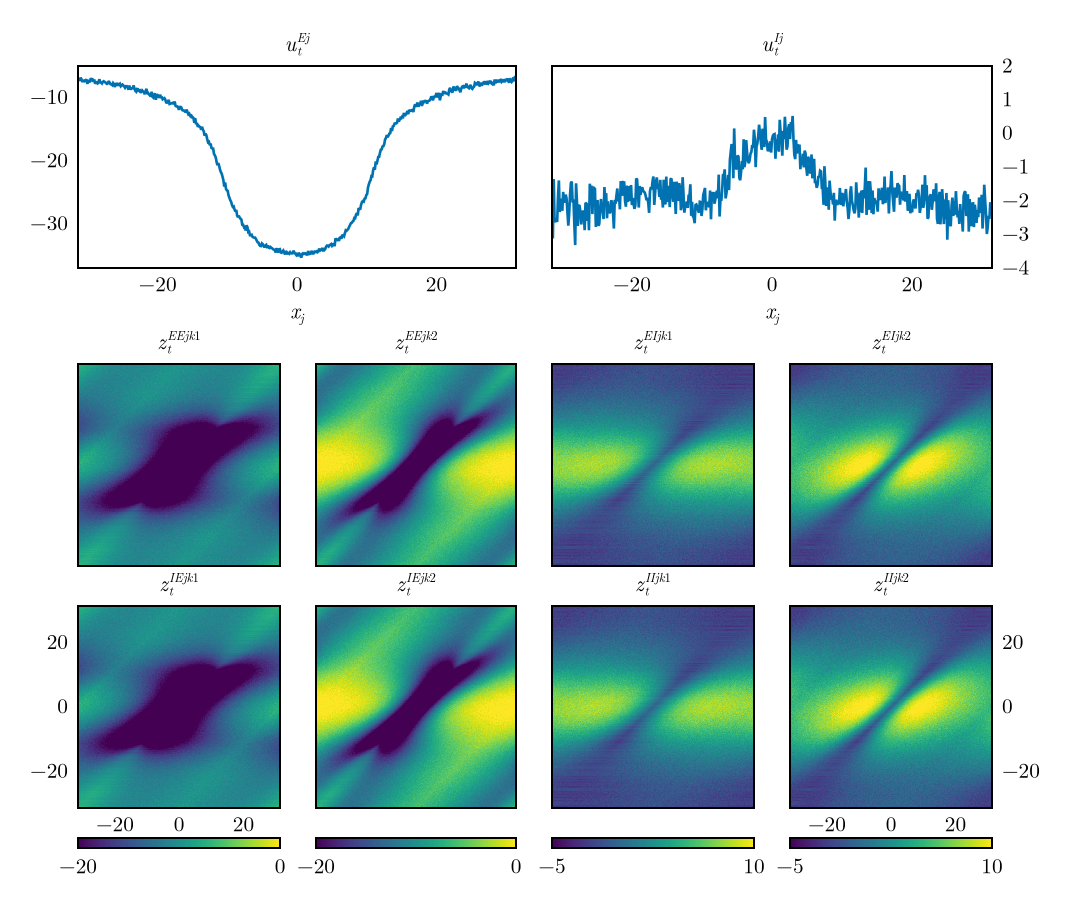}
  \caption{The solution profiles of the particle system in
    \cref{fig:oscillatory_pattern} at $t = 10$. This is the time at which
    weak convergence of the particle model to the mean-field model is tested in
  \cref{fig:weak_error_convergence}.}
  \label{fig:weak_error_profiles}
\end{figure}
\begin{figure}
  \centering
  \includegraphics{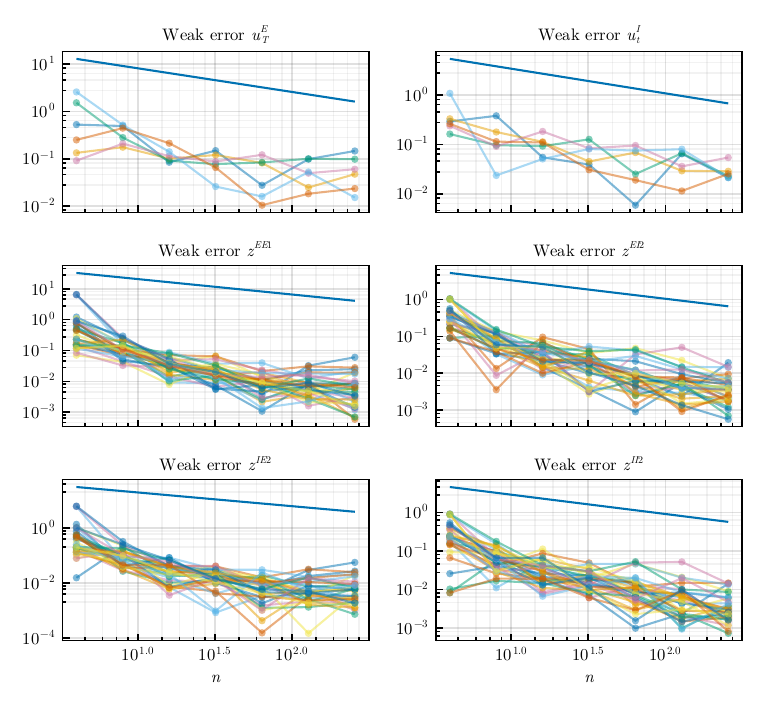}
  \caption{Weak error between the mean-field and particle solution at $t_*
    =10$ (see \cref{fig:weak_error_profiles}), for various values of $n$ (an
    $O(n^{-2})$ line is plotted in solid blue for reference). 
    The weak error is an integral of the absolute value of
    the difference between spatial profiles against $e^{i k_x x}$ (for the components
    $u^{\alpha}(x, t_*)$) and against $e^{i k_x x + i k_y y}$ for the components
    $z^{\alpha \beta r}(x,y, t_*)$ (see \cite[Section 5.3]{avitabile2026neural} for a
    definition of the weak error for univariate profiles). The error lines are
    plotted for $k_x, k_y \in \{ 0, 1,\ldots,5\}$}.
  \label{fig:weak_error_convergence}
\end{figure}

\section{Proof Outline} \label{Section Proof Outline}

In this section we sketch the outline of the proof. The proof details are deferred to Section 6. The existence and uniqueness of the limiting probability law $\mu$ is proved in Section \ref{Section Existence Uniqueness of Limiting Law}. 

The proof of Theorem \ref{Theorem main} (the main result of this paper) is complicated by the fact that the intrinsic dynamics given by $f$ and $g$ only satisfies a one-sided Lipschitz condition. We need to control the regularity of the paths. To this end, for any $a > 0$, define
\begin{align}
\mathcal{A}_a \subset \mathcal{C}\big( [ 0,T] , \mathbb{R}^d \big) \times \mathcal{C}\big( [ 0,T] , \mathbb{R}^c \big) \label{eq: A a definition}
\end{align}
to consist of all $(z,u)$ such that
\begin{align}
\sup_{t\leq T} \norm{z_t} &\leq a \\
\sup_{t\leq T} \norm{u_t} &\leq a \\
\sup_{s \neq t} \big\lbrace |t-s|^{-1/4} \norm{ z_t - z_s} \big\rbrace &\leq a \\
\sup_{s \neq t}\big\lbrace |t-s|^{-1/4} \norm{ u_t - u_s} \big\rbrace &\leq a .
\end{align}
\begin{lemma} \label{Lemma Regularity one}
For any $\epsilon > 0$, there exists $a_{\epsilon} > 0$ such that $\mathbb{P}$-almost-surely,
\begin{align}
\lsup{n} n^{-2} \sum_{j,k\in \mathbb{N}_n} \mathbf{1}\big\lbrace (z^{jk} , u^k) \notin \mathcal{A}_{a_{\epsilon}} \big\rbrace &\leq \epsilon \label{eq: first identity in regularity lemma main} \\
\mu\big(\mathcal{E} \times \mathcal{E} \times \mathcal{A}_{a_\epsilon} \big)& \leq \epsilon \label{eq: second identity in regularity lemma main}
\end{align}
\end{lemma}
The implication of Lemma \ref{Lemma Regularity one} is that we only need to prove the convergence with respect to a weaker topology. %(the Wasserstein metric induced by the $L^2$-norm). To this end, %let the $L^2$ norm on $\mathcal{E} \times \mathcal{E} \times  L^2\big( [0,T], \mathbb{R}^d  \big) \times L^2\big( [0,T],\mathbb{R}^c \big)$ be
%\begin{align}
%\norm{ (\theta,\eta,z,u)}_{L^2}^2 = \| \theta\|^2  + \| \eta \|^2 +  \int_0^T \| z_s \|^2 ds  +  \int_0^T \| u_s \|^2 ds  .
%\end{align}
Let the $L^2$ Wasserstein $2$-Distance on $\mathcal{P}\big( \mathcal{E} \times \mathcal{E} \times  L^2\big( [0,T], \mathbb{R}^d  \big) \times L^2\big( [0,T],\mathbb{R}^c \big) \big)$ be $d_{L^2}$, i.e.
\begin{align}
d_{L^2}(\mu,\nu) = \inf_{\zeta}\mathbb{E}^{\zeta}\bigg[ \bigg(  d_{\mathcal{E}}(\theta,\tilde{\theta}) + d_{\mathcal{E}}(\eta,\tilde{\eta}) + \norm{z - \tilde{z}}_{L^2} + \norm{ u-\tilde{u}}_{L^2} \bigg)^2 \bigg]^{1/2},
\end{align}
and the infimum is over all couplings $\zeta$ of $\mu$ and $\nu$.

The implication of Lemma \ref{Lemma Regularity one} is that we only need to prove convergence with respect to $d_{L^2}$, as noted in the following Lemma.
\begin{lemma} \label{Lemma Main Comparison Disordered Graph}
Suppose that with unit probability,
\begin{align} \label{eq: Assumption L squared convergence}
\lim_{n\to \infty}d_{L^2}\big( \hat{\mu}^n  , \mu \big) = 0.
\end{align}
Then with unit probability,
\begin{align}
    \lim_{n\to \infty}d_{W}\big( \hat{\mu}^n  , \mu \big) = 0 .
\end{align}
\end{lemma}
It remains for us to prove that $\lim_{n\to \infty}d_{L^2}\big( \hat{\mu}^n  , \mu \big) = 0$. To this end, we now specify a system of SDEs $\big\lbrace \bar{u}^j_{[0,T]} , \bar{z}^{jk}_{[0,T]} \big\rbrace_{j,k \in \mathbb{N}_n}$ with more homogeneous coupling than the original heterogeneous system. The initial conditions are identical, i.e.
\begin{align}
\bar{u}^j_0 = u^j_0 \text{ and } \bar{z}^{jk}_0 = z^{jk}_0.
\end{align}
 The dynamics of the nodes is a `mean-field' function of the edges, i.e.
\begin{align}
d\bar{u}^j_t = \bigg( f(\bar{u}^j_t) +I(t, x^j_n) +  n^{-1}\sum_{k\in \mathbb{N}_n} \mathcal{K}(x^j_n , x^k_n)  F (x^j_n , x^k_n, \bar{u}^j_t, \bar{z}^{jk}_t) \bigg) dt + \sigma(x^j_n, \bar{u}^j_t )dW^j_t. \label{eq: homogenzied 1}
\end{align}
 The dynamics of the $k\mapsto j$ edge is a function of the state variables at node $k$, i.e.
\begin{align}
d\bar{z}^{jk}_t =  G  (x^j_n, x^k_n ,  \bar{z}^{jk}_t, \bar{u}^k_t   )  dt + \gamma (x^j_n , x^k_n , \bar{z}^{jk}_t,  \bar{u}^k_t   ) dW^{jk}_t.\label{eq: homogenzied 2}
\end{align}
Now define the empirical measure generated by the above system,
\begin{align}
\bar{\mu}^n &=  n^{-2} \sum_{j,k\in I_n  } \delta_{x^j_n, x^k_n,  \bar{z}^{jk}_{[0,T]} , \bar{u}^k_{[0,T]} } \in \mathcal{P}\big( \mathcal{E} \times \mathcal{E}  \times \mathcal{C}\big( [0,T], \mathbb{R}^d  \big) \times \mathcal{C}\big( [0,T],\mathbb{R}^c \big) \big) .
\end{align}
Through a coupling argument, we find that the empirical measures must concentrate at the same values.
\begin{lemma} \label{Lemma intermediate 1}
With unit probability,
    \begin{align}  
\lim_{n\to \infty}d_{L^2}\big( \hat{\mu}^n  , \bar{\mu}^n \big) = 0.
\end{align}
\end{lemma}
Next we must show that the empirical measure generated by the system with homogenized coupling converges to the limit law.
\begin{lemma}\label{Lemma intermediate 6}
With unit probability,
    \begin{align}  
\lim_{n\to \infty}d_{L^2}\big( \mu  , \bar{\mu}^n \big) = 0.
\end{align}
\end{lemma}
\begin{proof}
This follows from Lemmas \ref{Lemma grave mu convergence} and \ref{Lemma grave mu bar mu} (below).
\end{proof}
To prove Lemma \ref{Lemma intermediate 6}, we define another intermediate high-dimensional system with frozen (non-random) interaction. That is, we define $\grave{u}^j \in \mathcal{C}([0,T],\mathbb{R}^c) $ and $\grave{z}^{jk} \in \mathcal{C}([0,T],\mathbb{R}^d) $ to be the strong solution of the SDEs
\begin{align}
d\grave{u}^j_t =& \bigg( f(\grave{u}^j_t) +I(t, x^j_n) + H (x^j_n , \grave{u}^j_t ,\mu_{x^j_n,t})  \bigg) dt + \sigma(x^j_n, \grave{u}^j_t )dW^j_t\label{eq: homogenzied 1 approx}\\
d\grave{z}^{jk}_t =&  G  (x^j_n, x^k_n , \grave{z}^{jk}_t,  \grave{u}^k_t  )  dt + \gamma (x^j_n , x^k_n, \grave{z}^{jk}_t  ,   \grave{u}^k_t  ) dW^{jk}_t,\label{eq: homogenzied 2 approx}
\end{align}
and we define $\grave{u}^j_0 = u^j_0$ and $\grave{z}^{jk}_0 = z^{jk}_0$. The function $H$ is defined in \eqref{eq: mu specification 3}. Also, recall that $\mu_{x^j_n,t} \in \mathcal{P}(\mathcal{E} \times \mathbb{R}^d \times \mathbb{R}^c)$ is the marginal conditional distribution of the limiting law $\mu$, defined immediately after Theorem \ref{Theorem main}. Now define the empirical measure generated by the above system,
\begin{align}
\grave{\mu}^n &=  n^{-2} \sum_{j,k\in I_n  } \delta_{x^j_n, x^k_n,  \grave{z}^{jk}_{[0,T]} , \grave{u}^k_{[0,T]} } \in \mathcal{P}\big( \mathcal{E} \times \mathcal{E}  \times \mathcal{C}\big( [0,T], \mathbb{R}^d  \big) \times \mathcal{C}\big( [0,T],\mathbb{R}^c \big) \big) ,
\end{align}
and define the marginal empirical measure relative to the $j^{th}$ neuron, i.e.
\begin{align}
\grave{\mu}^{n,j} &= n^{-1} \sum_{k\in I_n} \delta_{  x^k_n,  \grave{z}^{jk}_{[0,T]} , \grave{u}^k_{[0,T]} } \in  \mathcal{P}\big( \mathcal{E}   \times \mathcal{C}\big( [0,T], \mathbb{R}^d  \big) \times \mathcal{C}\big( [0,T],\mathbb{R}^c \big) \big).
\end{align}
In fact $\big\lbrace \grave{u}^j_{[0,T]} \big\rbrace_{j\in \mathbb{N}_n}$ are independent (because they have frozen coupling), and this means that $\grave{z}^{jk}_{[0,T]}$ is independent of $\grave{z}^{rs}_{[0,T]}$ if $k \neq s$. These strong independence properties help us to prove the following lemma.
\begin{lemma} \label{Lemma grave mu convergence}
$\mathbb{P}$-almost-surely,
\begin{align}
\lim_{n\to\infty} \sup_{j\in I_n} d_{L^2}\big( \grave{\mu}^{n,j} , \mu_{x^j_n} \big) &= 0 \\
\lim_{n\to\infty} d_{L^2}\big( \grave{\mu}^n , \mu \big) &= 0 .
\end{align}
\end{lemma}
It turns out that the marginal empirical measures $\big\lbrace \hat{\mu}^{n,j} \big\rbrace_{j\in \mathbb{N}_n}$ converge uniformly, and this helps us prove that $\hat{\mu}^n$ converges.
\begin{lemma}\label{Lemma grave mu bar mu}
$\mathbb{P}$-almost-surely,
\begin{align}
\lim_{n\to\infty} \sup_{j\in I_n} d_{L^2}\big( \bar{\mu}^{n,j} , \grave{\mu}^{n,j} \big) &= 0 \\
\lim_{n\to\infty} d_{L^2}\big( \bar{\mu}^n , \grave{\mu}^n \big) &= 0 .
\end{align}
\end{lemma}

We can now outline the proof of Theorem \ref{Theorem main}.
\begin{proof}
Thanks to Lemmas \ref{Lemma intermediate 1} and \ref{Lemma intermediate 6}, with unit probability,
    \begin{align}  
\lim_{n\to \infty}d_{L^2}\big( \hat{\mu}^n  , \mu \big) = 0.
\end{align}
But Lemma \ref{Lemma Main Comparison Disordered Graph} now implies that 
    \begin{align}  
\lim_{n\to \infty}d_{W}\big( \hat{\mu}^n  , \mu \big) = 0.
\end{align}
\end{proof}

\section{Proof Details} \label{Section Proof Details}
 There are four subsections to the Proof Details. In Subsection \ref{Section Existence Uniqueness of Limiting Law} we prove the existence and uniqueness of the limiting probability law; in Subsection \ref{Section Controlling the Fluctuations} we obtain bounds on the Holder Continuity of the paths; in Subsection \ref{Section Averaged Coupling} we use a coupling argument to prove that the original system must converge to the same limit as the system with homogenized interactions; and in Section \ref{Section Coupling} we use a coupling argument to prove that the system with homogenized interactions converges to the limiting law.
\subsection{Existence and Uniqueness of the Limiting Law} \label{Section Existence Uniqueness of Limiting Law}
In this section we prove that the limiting law $\mu$ is uniquely well-defined. We first recap its definition. For any $\eta,\theta \in \mathcal{E}$, define the stochastic processes $z_{\eta\theta,t} \in \mathbb{R}^d$, $u_{\eta\theta,t} \in \mathbb{R}^c$ to be the strong solution of the system
\begin{align}
dz_{\eta\theta,t} =& G\big( \eta,\theta,z_{\eta\theta,t}, u_{\eta\theta,t}  \big) dt + \gamma( \eta,\theta,z_{\eta\theta,t},u_{\eta\theta,t} ) dW_{\eta\theta,t} \label{eq: J plus limit 3} \\
du_{\eta\theta,t} =& \big( f(u_{\eta\theta,t}) + I(\theta,t) +  H (\theta, u_{\eta\theta,t} ,\mu_{\theta,t}) \big) dt + \sigma( \theta ,  u_{\eta\theta,t} )dW_{\theta,t}, \label{eq: zlimit 3}
\end{align}
where $\lbrace W_{\eta\theta,t}  , W_{\theta,t} \rbrace_{\eta,\theta \in \mathcal{E}}$ are independent Brownian Motions taking values in (respectively) $\mathbb{R}^d$ and $\mathbb{R}^c$, $\mu_{\eta,t} \in \mathcal{P}\big( \mathcal{E}  \times \mathbb{R}^d \times \mathbb{R}^c \big)$ is such that for any measurable sets $A_1,A_2 \subseteq \mathcal{E}$, and any measurable $B_1 \subseteq \mathbb{R}^d$ and $B_2 \subseteq \mathbb{R}^c$,
\begin{align}
\mu_{\eta,t}\big( A_2 \times B_1 \times B_2 \big) = \int_{A_2} \mathbb{P}\bigg( z_{\eta\theta,t} \in B_1 , u_{\eta\theta,t} \in B_2 \bigg)  \mu_{\mathcal{E}}(d\theta),
\end{align}
and
\begin{align}
H: \mathcal{E} \times \mathbb{R}^c \times \mathcal{P}\big(   \mathcal{E} \times \mathbb{R}^d \times \mathbb{R}^c \big) &\mapsto \mathbb{R}^d \\
 H(\theta,v,\mu) &=  \mathbb{E}^{( \beta,z,u) \sim \mu }\big[    \mathcal{K}(\theta,\beta)F(\theta,\beta,v, z  ) \big] .  \label{eq: H definition 3}
 \end{align} 
The initial conditions  $\big(z_{\eta\theta,0} ,  u_{\eta\theta,0} \big) $ are distributed according to $q_{\eta\theta} \in \mathcal{P}\big( \mathbb{R}^d \times \mathbb{R}^c \big)$ . We assume that  $\big(z_{\eta\theta,0} ,  u_{\eta\theta,0} \big) $ is independent of $\big(z_{\alpha\beta,0} ,  u_{\alpha\beta,0} \big) $ if either $\alpha\neq \eta$ or $\beta \neq \theta$.
\begin{lemma}
There exist unique stochastic processes that solve the system \eqref{eq: J plus limit 3} - \eqref{eq: H definition 3} and are such that
\begin{align}
\sup_{t\leq T} \sup_{\eta,\theta \in \mathcal{E}} \big\lbrace \mathbb{E}\big[ \| z_{\eta\theta,t} \|^2 + \| u_{\eta\theta,t} \|^2 \big] \big\rbrace < \infty
\end{align}
\end{lemma}
\begin{proof}
Define the space 
\[
\mathcal{X}_T \subset \mathcal{C}\big( [0,T] \times \mathcal{E} \times \mathcal{E} , \mathcal{P}\big(   \mathbb{R}^d \times \mathbb{R}^c \big)  \big)
\]
to consist of all $ \lbrace \nu_{t,\theta,\alpha} \rbrace$ such that, writing $\nu_{t,\theta , \alpha}$ to be the law of random variables $(z,u)$, it holds that \begin{align}
\sup_{t\leq T , \theta , \alpha \in\mathcal{E}} \mathbb{E}^{\nu_{t,\theta,\alpha}}\big[ \| z \|^2 + \|u \|^2 \big] < \infty .
\end{align}
We endow $\mathcal{X}_T$ with the metric $D_T$ that is such that
\begin{align}
D_{W,T}\big( \zeta,\xi \big) = \sup_{\eta,\theta\in \mathcal{E}}\sup_{t\leq T} d_W\big( \zeta_{t,\eta,\theta} , \xi_{t,\eta,\theta} \big),
\end{align}
where
\[
d_W: \mathcal{P}\big( \mathbb{R}^d \times \mathbb{R}^c \big) \times \mathcal{P}\big( \mathbb{R}^d \times \mathbb{R}^c \big) \mapsto \mathbb{R}^+
\]
is the Wasserstein metric,
\begin{align}
d_W(\mu,\nu) = \inf_{\zeta} \mathbb{E}^{\zeta}\big[  \| z - \tilde{z} \|^2 + \| u - \tilde{u} \|^2 \big]^{1/2}
\end{align}
where the law of $(z,u) $ is $\mu$, the law of $ (\tilde{z}, \tilde{u})$ is $\nu$, and the infimum is over all possible couplings of $\mu$ and $\nu$.

For any $\nu \in \mathcal{X}_T$, define the mapping $\Phi_T: \mathcal{X}_T \mapsto \mathcal{X}_T$ as follows. Write $\Phi_T(\nu) := \zeta$, where $\zeta := \big\lbrace \zeta_{t,\theta , \alpha} \big\rbrace_{t\leq T , \theta ,\alpha \in \mathcal{E}}$ is defined as follows. Define $\zeta_{t,\eta,\theta}$ to be the law of random variables $\big( \tilde{z}_{\eta\theta,t} , \tilde{u}_{\eta\theta,t} \big)$, which are strong solutions of the SDE
\begin{align}
d\tilde{z}_{\eta\theta,t} =& G\big( \eta,\theta , \tilde{z}_{\eta\theta,t} ,\tilde{u}_{\eta\theta,t}\big) dt + \gamma( \eta,\theta, \tilde{z}_{\eta\theta,t}), \tilde{u}_{\eta\theta,t}  dW_{\eta\theta,t} \label{eq: J plus limit 2} \\
d\tilde{u}_{\eta\theta,t} =& \bigg( f(\tilde{u}_{\eta\theta,t}) + I(\theta,t) +  \int_{\mathcal{E}}\mathcal{K}(\theta,\alpha) \mathbb{E}^{(\breve{z},\breve{u}) \sim \nu_{t,\eta,\theta}}\big[ F(\theta,\alpha, \tilde{u}_{\eta\theta,t} , \breve{z}) \big]\mu_{\mathcal{E}}(d\alpha) \bigg) dt \nonumber \\ &+ \sigma( \theta ,  \tilde{u}_{\eta\theta,t} )d\tilde{W}_{\eta\theta,t}, \label{eq: zlimit 2}
\end{align}
where $\lbrace W_{\eta\theta,t}  , \tilde{W}_{\eta\theta,t} \rbrace_{\eta,\theta \in \mathcal{E}}$ are independent Brownian Motions taking values in (respectively) $\mathbb{R}^d$ and $\mathbb{R}^c$. Since $f$ and $G$ satisfy one-sided Lipschitz properties, the unique strong solution is well-known \cite{Oksendahl2003}.

We must next demonstrate that $\Phi_T$ is a contraction on $\mathcal{X}_T$, for small enough $T$. For $\phi \in \mathcal{X}_T$, define the random variables $\big( \grave{z}_{\eta\theta,t} , \grave{u}_{\eta\theta,t} \big)$ to be the strong solutions of the SDE
\begin{align}
d\grave{z}_{\eta\theta,t} =& G\big( \eta, \theta , \grave{z}_{\eta\theta,t} , \grave{u}_{\eta\theta,t}\big) dt + \gamma( \eta,\theta,\grave{u}_{\eta\theta,t} , \grave{z}_{\eta\theta,t}) dW_{\eta\theta,t} \label{eq: J plus limit 4} \\
d\grave{u}_{\eta\theta,t} =& \bigg( f(\grave{u}_{\eta\theta,t}) + I(\theta,t) +  \int_{\mathcal{E}}\mathcal{K}(\theta,\alpha) \mathbb{E}^{(\breve{z},\breve{u}) \sim \phi_{\theta,\alpha,t}}\big[ F(\theta,\alpha, \grave{u}_{\eta\theta,t} , \breve{z}) \big]\mu_{\mathcal{E}}(d\alpha) \bigg) dt \nonumber \\  &+ \sigma( \theta ,  \grave{u}_{\eta\theta,t} )d\tilde{W}_{\eta\theta,t}, \label{eq: zlimit  4}
\end{align}
with initial condition $\grave{z}_{\eta\theta,0} = \tilde{z}_{\eta\theta,0}$ and $\grave{u}_{\eta\theta,0} = \tilde{u}_{\eta\theta,0}$. An application of Ito's Lemma implies that
\begin{multline}
d \big\| \grave{z}_{\eta\theta,t} - \tilde{z}_{\eta\theta,t} \big\|^2 = 2 \bigg\langle \grave{z}_{\eta\theta,t} - \tilde{z}_{\eta\theta,t} , G\big(\eta,\theta, \grave{z}_{\eta\theta,t} , \grave{u}_{\eta\theta,t} \big) - G\big(\eta,\theta , \tilde{z}_{\eta\theta,t} , \tilde{u}_{\eta\theta,t}\big) \bigg\rangle dt \\ +
2\bigg\langle  \grave{z}_{\eta\theta,t} - \tilde{z}_{\eta\theta,t} ,  \big( \gamma( \eta,\theta, \grave{z}_{\eta\theta,t},\grave{u}_{\eta\theta,t} ) - \gamma( \eta,\theta, \tilde{z}_{\eta\theta,t},\tilde{u}_{\eta\theta,t} )  \big) dW_{\eta\theta,t}   \bigg\rangle\\
+ \rm{tr}\bigg(  \big( \gamma^T( \eta,\theta, \grave{z}_{\eta\theta,t},\grave{u}_{\eta\theta,t} ) - \gamma^T( \eta,\theta, \tilde{z}_{\eta\theta,t},\tilde{u}_{\eta\theta,t} )  \big)   \big( \gamma( \eta,\theta, \grave{z}_{\eta\theta,t},\grave{u}_{\eta\theta,t}  ) - \gamma( \eta,\theta, \tilde{z}_{\eta\theta,t},\tilde{u}_{\eta\theta,t} )  \big) \bigg) dt
\end{multline}
and
\begin{multline}
d\big\| \grave{u}_{\eta\theta,t} - \tilde{u}_{\eta\theta,t} \big\|^2 = 2 \bigg\langle \grave{u}_{\eta\theta,t} - \tilde{u}_{\eta\theta,t} , f(\grave{u}_{\eta\theta,t}) -  f(\tilde{u}_{\eta\theta,t}) \\+  \int_{\mathcal{E}}\mathcal{K}(\theta,\alpha) \bigg\lbrace \mathbb{E}^{(\breve{z},\breve{u}) \sim \phi_{t,\eta,\theta}}\big[ F(\theta,\alpha, \grave{u}_{\eta\theta,t} , \breve{z}) \big]-  \mathbb{E}^{(\breve{z},\breve{u}) \sim \nu_{t,\eta,\theta}}\big[ F(\theta,\alpha, \tilde{u}_{\eta\theta,t} , \breve{z}) \big]  \bigg\rbrace \mu_{\mathcal{E}}(d\alpha) \bigg\rangle dt \\
+  2 \bigg\langle \grave{u}_{\eta\theta,t} - \tilde{u}_{\eta\theta,t} , \big( \sigma(\theta , \grave{u}_{\eta\theta,t} )- \sigma(\theta , \tilde{u}_{\eta\theta,t} ) \big) d\tilde{W}_{\eta\theta,t} \bigg\rangle \\
+ \rm{tr}\bigg(  \big( \sigma^T(\theta , \grave{u}_{\eta\theta,t} )- \sigma^T(\theta , \tilde{u}_{\eta\theta,t} ) \big)  \big( \sigma(\theta , \grave{u}_{\eta\theta,t} )- \sigma(\theta , \tilde{u}_{\eta\theta,t} ) \big)  \bigg) dt .
\end{multline}
Employing the Lipschitz assumptions, and taking expectations, we obtain that there is a constant $\tilde{C} > 0$ such that for any $t \leq T$,
\begin{align}
\mathbb{E}\bigg[ \big\| \grave{z}_{\eta\theta,t} - \tilde{z}_{\eta\theta,t} \big\|^2 \bigg] &\leq \tilde{C}  \int_0^t \mathbb{E}\bigg[ \big\| \grave{z}_{\eta\theta,s} - \tilde{z}_{\eta\theta,s} \big\|^2 +  \big\| \grave{u}_{\eta\theta,s} - \tilde{u}_{\eta\theta,s} \big\|^2  \bigg] ds \label{eq: grave z tilde z inequality} \\
\mathbb{E}\bigg[ \big\| \grave{u}_{\eta\theta,t} - \tilde{u}_{\eta\theta,t} \big\|^2  \bigg] &\leq \tilde{C} \int_0^t\bigg\lbrace \mathbb{E}\bigg[ \big\| \grave{u}_{\eta\theta,s} - \tilde{u}_{\eta\theta,s} \big\|^2  \bigg] + \sup_{\alpha,\beta \in \mathcal{E}} d_W\big(\phi_{\alpha\beta,s} , \nu_{\alpha\beta,s}  \big)^2 \bigg\rbrace ds. \label{eq: grave u tilde u inequality}
\end{align}
Now it is immediate from the definition of the Wasserstein Distance that
\begin{align}
D_{W,t}\big( \Phi_t(\nu) , \Phi_t(\phi) \big)^2 \leq \sup_{\eta,\theta\in \mathcal{E}} \sup_{s\leq t} \mathbb{E}\bigg[ \big\| \grave{z}_{\eta\theta,s} - \tilde{z}_{\eta\theta,s} \big\|^2 + \big\| \grave{u}_{\eta\theta,s} - \tilde{u}_{\eta\theta,s} \big\|^2  \bigg].  \label{eq: Wasserstein Inequality 1}
\end{align}
It thus follows from \eqref{eq: grave z tilde z inequality} - \eqref{eq: Wasserstein Inequality 1} that there is a constant $\bar{C} > 0$ such that
\begin{align}
D_{W,t}\big( \Phi_t(\nu) , \Phi_t(\phi) \big)^2 \leq \bar{C} t D_{W,t}\big( \nu , \phi \big)^2.
\end{align}
For small enough $t$, the mapping $\Phi_t$ is therefore a contraction, with a unique fixed point. We can then iterate this method, obtaining a unique fixed point $\lbrace \nu_{\eta\theta,t} \rbrace_{t\leq T \fatsemi \eta,\theta\in\mathcal{E}}$ for arbitrary $T$. 

We can now define $\mu \in \mathcal{P}\big( \mathcal{E} \times \mathcal{E} \times \mathcal{C}\big( [0,T], \mathbb{R}^d\big) \times \mathcal{C}\big( [0,T], \mathbb{R}^c \big) \big)$ to be the unique measure such that for measurable $A_1,A_2 \subseteq \mathcal{E}$ and measurable $B_1 \subseteq \mathcal{C}\big( [0,T], \mathbb{R}^d\big) $ and measurable $B_2 \subseteq \mathcal{C}\big( [0,T], \mathbb{R}^c\big) $,
\begin{align}
\mu\big( A_1 \times A_2 \times B_1 \times B_2 \big) = \int_{A_1} \int_{A_2} \mathbb{P}\bigg( \tilde{z}_{\eta\theta,[0,T]} \in B_1 ,  \tilde{u}_{\eta\theta,[0,T]} \in B_2 \bigg) \mu_{\mathcal{E}}(d\eta)\mu_{\mathcal{E}}(d\theta) 
\end{align}
 \end{proof}

\subsection{Controlling the Fluctuations in Time} \label{Section Controlling the Fluctuations}
The main goal of this section is to prove Lemmas \ref{Lemma Regularity one} and \ref{Lemma Main Comparison Disordered Graph}. In the first result of this section, we note that if one has uniform control over the Holder Continuity of a path, then the $L^2$ convergence implies $L^{\infty}$ convergence.

 \begin{lemma}\label{Lemma Bound L2 supremum}
For any  $b \geq T^{1/4}$, there is a constant $C_b > 0$ such that the following holds.  For any $u,\tilde{u} \in \mathcal{C}\big([0,T], \mathbb{R}^c\big)$ and $z,\tilde{z} \in \mathcal{C}\big([0,T], \mathbb{R}^d\big)$ such that 
 \begin{align}
    \sup_{h\leq 1} \sup_{0\leq t \leq T-h} \big\lbrace h^{-1/4} \norm{\tilde{u}_{t+h} - \tilde{u}_t} \big\rbrace &= b < \infty \\
   \sup_{h\leq 1} \sup_{0\leq t \leq T-h} \big\lbrace h^{-1/4} \norm{z_{t+h} - \tilde{z}_t} \big\rbrace &= b < \infty \\
 \norm{u_0}^2 &\leq b \\
  \norm{\tilde{u}_0}^2 &\leq b \\
  \norm{z_0}^2 &\leq b  \\
    \norm{\tilde{z}_0}^2 &\leq b,
 \end{align}
 it holds that
\begin{align}
\sup_{t\leq T}\big\| u(t) - \tilde{u}(t) \big\|^2 &\leq C_b\int_0^T \big\| u(t) - \tilde{u}(t) \big\|^2  dt \label{eq: u C_b inequality} \\
\sup_{t\leq T}\big\| z(t) - \tilde{z}(t) \big\|^2 &\leq C_b\int_0^T \big\| z(t) - \tilde{z}(t) \big\|^2  dt \label{eq: z C_b inequality} 
\end{align}
\end{lemma}
\begin{proof}
We prove \eqref{eq: u C_b inequality} (the proof of \eqref{eq: z C_b inequality}  is almost identical). Define
\begin{align}
C_b = \bigg( \int_0^{b^{-4}} \big( 1 - bt^{1/4} \big) dt \bigg)^{-1} .
\end{align}
Let $s \in [0,T]$ be such that
\begin{align}
\sup_{t\leq T}\big\| u(t) - \tilde{u}(t) \big\| = 
\big\| u(s) - \tilde{u}(s) \big\|.
\end{align}
Suppose first that $s \leq T/2$. Define
\begin{align}
\tau = T\wedge \inf\bigg\lbrace t \geq s  : \big\| u(t) - \tilde{u}(t) \big\| = \frac{1}{2}\big\| u(s) - \tilde{u}(s) \big\| \bigg\rbrace .
\end{align}
Then necessarily 
\begin{equation}
\tau-s \geq T/2 \wedge
\bigg(\frac{1}{2b} \big\| u(s) - \tilde{u}(s) \big\|\bigg)^4,
\end{equation}
and therefore
\begin{align}
\int_s^{\tau} \big\| u(t) - \tilde{u}(t) \big\|^2 dt \geq \frac{1}{2}\big\| u(s) - \tilde{u}(s) \big\| \times \bigg\lbrace \bigg(\frac{1}{2b} \big\| u(s) - \tilde{u}(s) \big\|\bigg)^4 \wedge T/2 \bigg\rbrace
\end{align}
%Then choosing
%\begin{align}
%\big\| u^j_r - \tilde{u}_r \big\| \geq \big\| u(s) - \tilde{u}(s) \big\| - \big\| u(s) - u(r) \big\| + \big\| \tilde{u}(s) - \tilde{u}(r) \big\|  \\
%\geq \big\| u(s) - \tilde{u}(s) \big\| - bh^{1/4} - bh^{1/4}
%\end{align}
The case $s   > T/2$ is handled similarly.
\end{proof}

%The proof of Lemma \ref{Lemma Main Comparison Disordered Graph} is complicated by the fact that (by assumption) the intrinsic dynamics given by $f$ and $g$ is not Lipschitz. We will start by proving the convergence with respect to a weaker topology (the Wasserstein metric induced by the $L^2$-norm.) To this end, 
%Let the $L^2$ norm on $\mathcal{E} \times \mathcal{E} \times  L^2\big( [0,T], \mathbb{R}^d  \big) \times L^2\big( [0,T],\mathbb{R}^c \big)$ be
%\begin{align}
%\norm{ (\theta,\eta,z,u)}_{L^2}^2 = \| \theta\|^2  + \| \eta \|^2 +  \int_0^T \| z_s \|^2 ds  +  \int_0^T \| u_s \|^2 ds  .
%\end{align}
%Let the induced Wasserstein Distance on $\mathcal{E} \times \mathcal{E} \times  L^2\big( [0,T], \mathbb{R}^d  \big) \times L^2\big( [0,T],\mathbb{R}^c \big)$ be $d_{L^2}$, i.e.
%\begin{align}
%d_{L^2}(\mu,\nu) = \inf_{\zeta}\mathbb{E}^{\zeta}\bigg[ \norm{ (\theta-\tilde{\theta} , \eta-\tilde{\eta} , z - \tilde{z} , u-\tilde{u})}^2 \bigg]^{1/2},
%\end{align}
%and the infimum is over all couplings $\zeta$ of $\mu$ and $\nu$.
%Next, to compare the empirical measures $\hat{\mu}^n$ and $\bar{\mu}^n$ , we make the `obvious' coupling by matching the same indices.
We next define the stopping times, for $a > 0$,
\begin{align}
\tau^j_{a} &= \inf\bigg\lbrace t \leq T \; : \; \norm{\bar{u}^j_t} = a\bigg\rbrace \\
\tau^{jk}_{a} &= \inf\bigg\lbrace t \leq T \; : \; \norm{\bar{z}^{jk}_t} = a\bigg\rbrace 
\end{align}
and the $\lbrace 0,1 \rbrace$-valued random variables, for $b > 0$,
\begin{align}
r^j_a =& \mathbf{1}\big\lbrace \tau^j_a < T \big\rbrace \\
r^{jk}_a =& \mathbf{1}\big\lbrace \tau^{jk}_a < T \big\rbrace \\
q^j_{a,b  } =& \mathbf{1}\bigg\lbrace \tau^j_a = T \text{ and } \sup_{h \leq \epsilon}\sup_{t \leq T-h}  h^{-1/4}\norm{\bar{u}^j_{t+h} - \bar{u}^j_t} \geq b \bigg\rbrace \\
q^{jk}_{a,b} =& \mathbf{1}\bigg\lbrace \tau^{jk}_a = T \text{ and } \sup_{h \leq \epsilon}\sup_{t \leq T-h}  h^{-1/4}\norm{\bar{z}^{jk}_{t+h} - \bar{z}^{jk}_t} \geq b \bigg\rbrace .
\end{align}
We must first demonstrate that the number of neurons / synapses whose Euclidean norm exceeds a threshold can be controlled.
\begin{lemma} \label{eq: temporary r lemma}
For any $\epsilon > 0$, there exists $a(\epsilon)$ such that for all $a \geq a(\epsilon)$,
\begin{align}
 \lsup{n} n^{-1} \sum_{j \in \mathbb{N}_n}r^j_a  &\leq \epsilon \label{eq: to prove fluctuations bound 1}\\
 \lsup{n} n^{-2} \sum_{j,k \in \mathbb{N}_n}r^{jk}_a   &\leq \epsilon .\label{eq: to prove fluctuations bound 2}
\end{align}
\end{lemma}
\begin{proof}
We prove \eqref{eq: to prove fluctuations bound 1} (the proof of \eqref{eq: to prove fluctuations bound 2} is very similar). Thanks to Ito's Lemma,
\begin{multline}
 \big\| u^j_T \big\|^2 = \| u^j_0 \| ^2 + 2 \int_0^T \bigg\langle u^j_t ,  f(u^j_t) +I(t, x^j_n) +  n^{-1}\sum_{k\in \mathbb{N}_n}  \varphi_n^{-1} K_n^{jk}  F (x^j_n , x^k_n, u^j_t, z^{jk}_t)  \bigg\rangle dt\\  
+ \int_0^T \rm{tr}\big( \sigma(u^j_t) \sigma(u^j_t)^T \big) dt +
 2\int_0^T\langle u^j_t , \sigma(u^j_t)dW^j_t \rangle .
\end{multline}
Let $0\leq \alpha_j < \beta_j \leq T$ be any two stopping times such that $\mathbb{P}$-almost-surely, for all $t\in [\alpha_j,\beta_j]$
\begin{align}
\big\| u^j_{\alpha_j} \big\| &= 1 \text{ or }\alpha_j = T \\
\beta_j &= \inf\big\lbrace t\in [\alpha_j,T] \; : \; \big\| u^j_{t} \big\| = 1/2 \big\rbrace .
\end{align}
Taking square roots, Ito's Lemma implies that for all $t \in [\alpha_j,\beta_j]$,
\begin{multline}
 d\big\| u^j_t \big\|  =  \big\| u^j_t \big\|^{-1} \bigg\langle u^j_t ,  f(u^j_t) +I(t, x^j_n) +  n^{-1}\sum_{k\in \mathbb{N}_n}  \varphi_n^{-1} K_n^{jk}  F (x^j_n , x^k_n, u^j_t, z^{jk}_t)  \bigg\rangle dt \\
+\frac{1}{2} \big\| u^j_t \big\|^{-1} \rm{tr}\big( \sigma(u^j_t) \sigma(u^j_t)^T \big)dt
-\frac{1}{2} \big\| u^j_t \big\|^{-3}\big\| \sigma(u^j_t)^T u^j_t \big\|^2 dt
 +   \big\| u^j_t \big\|^{-1} \langle u^j_t , \sigma(u^j_t)dW^j_t \rangle .
\end{multline}
Our assumptions on the regularity of the functions imply that there is a constant $C > 0$ such that, for all  $ t\in [\alpha_j,\beta_j]$,
\begin{equation}
 d\big\| u^j_t \big\|  \leq C\big(\big\| u^j_t \big\| + 1 + Q_n^j \big)dt  +   \big\| u^j_t \big\|^{-1} \langle u^j_t , \sigma(u^j_t)dW^j_t \rangle . \label{eq: start list eqs}
\end{equation}
It thus follows from Gronwall's Inequality that
\begin{align}
\sup_{t\in [\alpha_j,\beta_j]}\big\| u^j_t \big\|  \leq  \exp\big( CT \big) \bigg(  CT+ TCQ^j_n+ 2 \sup_{t\in [\alpha_j,\beta_j]}\bigg\lbrace \int_{\alpha_j}^t  \big\| u^j_s \big\|^{-1}\big\langle u^j_s , \sigma(u^j_s)dW^j_s \big\rangle \bigg\rbrace \bigg).
\end{align}
Since $\sigma$ is uniformly bounded, there exists a non-random constant $\tilde{C} > 0$ such that the quadratic variation of the stochastic integral
\[
M^j_T = \int_0^T \big\| u^j_t \big\|^{-1} \langle u^j_t , \sigma(u^j_t)dW^j_t \rangle 
\]
is upperbounded by $\tilde{C}T$. We thus find that
\begin{equation}
\bigg\lbrace 2 \sup_{t\in [\alpha_j,\beta_j]}\bigg\lbrace \int_{\alpha_j}^t  \big\| u^j_s \big\|^{-1}\big\langle u^j_s , \sigma(u^j_s)dW^j_s \big\rangle \bigg\rbrace  \geq a \bigg\rbrace \subseteq \mathcal{U}_j,
\end{equation}
where 
\begin{align}
\mathcal{U}_j = \bigg\lbrace \text{ There exist }s < t \leq T \text{ such that }w_j(\tilde{C}t)- w_j(\tilde{C}s) \geq a   \bigg\rbrace ,
\end{align}
and $w_j(t)$ is such that
\[
M^j(t) = w_j\bigg( \int_0^t \big\| u^j_s \big\|^{-2} \big\| \sigma(u^j_t)^T u^j_t \big\|^2 ds \bigg).
\]
It is known that $w_j(t)$ has the same probability law as a standard one-dimensional Brownian Motion \cite{Karatzas1991}. Furthermore, $w_j(t)$ is independent of $w_k(t)$ if $j\neq k$.

We thus find that for any positive constant $a>0$,
\begin{equation}
 \mathbb{P}\bigg(\sup_{t\in [0,T]}\big\| u^j_t \big\| \geq  \exp\big( CT \big) \big(a+  CT+ TCQ^j_n \big) \bigg)   
 \leq \mathbb{P}\big( \mathcal{U}_j \big).
\end{equation}
We furthermore find that there is a constant $C_2 > 0$ such that for all $a \geq 1$
\begin{multline}
\mathbb{P}\bigg( \text{ There exists }s < t \leq T \text{ such that }w_j(\tilde{C}t)- w_j(\tilde{C}s) \geq a \bigg)
\leq \\ \mathbb{P}\bigg( \sup_{t\leq T}w_j(\tilde{C}t) \geq \frac{a}{2} \bigg)  + \mathbb{P}\bigg( \inf_{s\leq T} w_j(\tilde{C}s) \leq -\frac{a}{2} \bigg)
\leq  C_2 a^{-2},
\end{multline}
by a standard property of Brownian Motion. By assumption,
\begin{equation}
 \lim_{n\mapsto \infty}n^{-1}\sum_{j\in \mathbb{N}_n} Q_n^j=  0. \label{eq: finish list eqs}
\end{equation} It now follows from \eqref{eq: start list eqs}-\eqref{eq: finish list eqs} that
\begin{align}
 \lim_{a\to \infty} \lsup{n} n^{-1} \mathbb{E}\bigg[ \sum_{j \in \mathbb{N}_n}r^j_a \bigg] = 0.
\end{align}
Since the events $\big\lbrace \mathcal{U}_j \big\rbrace_{j\in \mathbb{N}_n}$ are independent, \eqref{eq: to prove fluctuations bound 1} now follows from Bernstein`'s Inequality.
\end{proof}
\begin{lemma} \label{Lemma bound the number of q}
Let $a > 0$. If $b(a)$ is sufficiently large, then for all $b\geq b(a)$, it holds $\mathbb{P}$-almost-surely that
\begin{align}
 \lsup{n} n^{-1} \sum_{j \in \mathbb{N}_n} q^{j}_{a,b} &\leq \epsilon \\
 \lsup{n} n^{-2} \sum_{j,k \in \mathbb{N}_n} q^{jk}_{a,b} &\leq \epsilon .
\end{align}
\end{lemma}
The proof is very similar to that of Lemma \ref{eq: temporary r lemma} and is neglected. It uses the fact that the advective functions are locally Lipschitz on bounded sets.
%We now prove Lemma \ref{Lemma bound the number of q}.

%\begin{proof}
%Define the random index set $\mathcal{I}^{(a)}_n \subseteq \mathbb{N}_n$ to consist of all $j$ such that $\tau^j_a <T$ and define the random index set $\tilde{\mathcal{I}}^{(a)}_n \subseteq \mathbb{N}_n \times \mathbb{N}_n$ to consist of all $(j,k)$ such that $\tau^{jk}_a  < T$. 
%Define the random index set $\mathcal{I}^{(a,\epsilon)}_n \subseteq \mathbb{N}_n$ to consist of all $j$ such that $\tau^j_a =T$ and $q^j_{a,\epsilon} = 1$. Define the random index set $\tilde{\mathcal{I}}^{(a,\epsilon)}_n \subseteq \mathbb{N}_n \times \mathbb{N}_n$ to consist of all $(j,k)$ such that $\tau^{jk}_a =T$ and $q^{jk}_{a,\epsilon} = 1$. Finally define the random index set $\breve{\mathcal{I}}^{(a,\epsilon)}_n \subseteq \mathbb{N}_n $ to consist of all $j\in \mathbb{N}_n$ that are not in $\mathcal{I}^{(a)}_n \cup \mathcal{I}^{(a,\epsilon)}_n$, and define the random index set $\grave{\mathcal{I}}^{(a,\epsilon)}_n \subseteq \mathbb{N}_n \times \mathbb{N}_n$ to consist of all $(j,k)\in \mathbb{N}_n \times \mathbb{N}_n$ that are not in $\tilde{\mathcal{I}}^{(a)}_n \cup \tilde{\mathcal{I}}^{(a,\epsilon)}_n$.
%\end{proof}
We can now prove Lemma \ref{Lemma Regularity one}.
\begin{proof}
\eqref{eq: second identity in regularity lemma main} follows from the fact that
\begin{align}
\sup_{t\leq T} \sup_{\eta,\theta\in\mathcal{E}} \mathbb{E}\big[ \big\| z_{\eta\theta,t} \big\|^2 + \big\| u_{\eta\theta,t} \big\|^2  \big] < \infty.
\end{align}
It remains to prove \eqref{eq: first identity in regularity lemma main}, i.e. we must prove that for any $\epsilon > 0$, for large enough $a_{\epsilon}$, it holds that $\mathbb{P}$-almost-surely,
\begin{align}
\lsup{n} n^{-2} \sum_{j,k\in \mathbb{N}_n} \mathbf{1}\big\lbrace (z^{jk} , u^k) \notin \mathcal{A}_{a_{\epsilon}} \big\rbrace &\leq \epsilon . \label{eq: first identity in regularity lemma main restated} 
\end{align}
 Now as long as $a$ is large enough,
 \begin{multline}
\lsup{n}  n^{-2} \sum_{j,k\in \mathbb{N}_n} \mathbf{1}\big\lbrace (z^{jk} , u^k) \notin \mathcal{A}_{a} \big\rbrace
  \leq \\
\lsup{n} n^{-1} \sum_{j \in \mathbb{N}_n}\big( r^j_a + q^{j}_{a,a} \big)  +\lsup{n} n^{-2} \sum_{j,k \in \mathbb{N}_n}\big( r^{jk}_a + q^{jk}_{a,a} \big) 
 \leq \epsilon,
   \end{multline}
$\mathbb{P}$-almost-surely, as a consequence of Lemmas \ref{eq: temporary r lemma} and \ref{Lemma bound the number of q}.
\end{proof}

We can now prove Lemma \ref{Lemma Main Comparison Disordered Graph}.
\begin{proof}
Let 
\[
\zeta_n  \in \mathcal{P}\bigg( \mathcal{E} \times \mathcal{E} \times \mathcal{C}\big([0,T],\mathbb{R}^d\big) \times \mathcal{C}\big([0,T],\mathbb{R}^c\big)  \times \mathcal{E} \times \mathcal{E} \times \mathcal{C}\big([0,T],\mathbb{R}^d\big) \times \mathcal{C}\big([0,T],\mathbb{R}^c\big)\bigg)
\]
be any coupling of $\hat{\mu}^n$ and $\mu$ that is within $n^{-1}$ of realizing the infimum in the definition of the Wasserstein distance $d_{L^2}$. Write $\zeta_n$ to be the law of the random variables $(\theta,\eta,z,u)$ and $(\tilde{\theta}, \tilde{\eta}, \tilde{z}, \tilde{u})$. We thus have that
\begin{align*}
d_{L^2}\big( \hat{\mu}^n  , \mu \big)^2 \geq \mathbb{E}^{\zeta_n}\bigg[ \bigg( d_{\mathcal{E}}(\theta , \tilde{\theta}) + d_{\mathcal{E}}(\eta ,\tilde{\eta}) +  \norm{z -\tilde{z}}_{L^2} +  \norm{u - \tilde{u}}^2_{L^2} \bigg)^2 \bigg] - n^{-1},
\end{align*}
and therefore (as a consequence of our assumption in the statement of the Lemma, i.e. \eqref{eq: Assumption L squared convergence}),
\begin{align}
\lim_{n\to\infty}\mathbb{E}^{\zeta_n}\bigg[ d_{\mathcal{E}}(\theta , \tilde{\theta}) + d_{\mathcal{E}}(\eta ,\tilde{\eta}) +  \norm{z -\tilde{z}}_{L^2} +  \norm{u - \tilde{u}}^2_{L^2} \bigg] = 0. \label{eq: zeta n zero}
\end{align}
Now for any $a \gg 1$, and recalling the definition of $\mathcal{A}_a$ in \eqref{eq: A a definition},
\begin{multline*}
 \mathbb{E}^{\zeta_n}\bigg[ 1 \wedge \bigg( \norm{z - \tilde{z}}_T +  \norm{u - \tilde{u}}_T \bigg) \bigg] 
 \leq \mathbb{E}^{\zeta_n}\bigg[ \mathbf{1}\big\lbrace (z,u) \notin \mathcal{A}_a \text{ or } (\tilde{z},\tilde{u}) \notin \mathcal{A}_a \big\rbrace  \bigg]  \\ +\mathbb{E}^{\zeta_n}\bigg[ \mathbf{1}\big\lbrace (z,u) \in \mathcal{A}_a \text{ and } (\tilde{z},\tilde{u}) \in \mathcal{A}_a \big\rbrace \bigg( \norm{z - \tilde{z}}_T +  \norm{u - \tilde{u}}_T \bigg)\bigg] .
\end{multline*}
Thanks to the Cauchy-Schwarz Inequality, 
\begin{align}
\mathbb{E}^{\zeta_n}\bigg[ \mathbf{1}\big\lbrace (z,u) \in \mathcal{A}_a \text{ and } &(\tilde{z},\tilde{u}) \in \mathcal{A}_a \big\rbrace \bigg( \norm{z - \tilde{z}}_T +  \norm{u - \tilde{u}}_T \bigg)\bigg] \nonumber \\
&\leq \sqrt{2}\mathbb{E}^{\zeta_n}\bigg[ \mathbf{1}\big\lbrace (z,u) \in \mathcal{A}_a \text{ and } (\tilde{z},\tilde{u}) \in \mathcal{A}_a \big\rbrace \bigg( \norm{z - \tilde{z}}^2_T +  \norm{u - \tilde{u}}_T^2 \bigg)\bigg]^{1/2} \nonumber \\
&\leq \sqrt{2 C_a} \mathbb{E}^{\zeta_n}\bigg[   \norm{z - \tilde{z}}^2_{L^2} +  \norm{u - \tilde{u}}_{L^2}^2 \bigg]^{1/2}. \label{eq: temporary L squared ineq}
\end{align}
for a constant $C_a > 0$, thanks to Lemma \ref{Lemma Bound L2 supremum}. It thus holds that for any $a > 0$,
\begin{align}
\lim_{n\to\infty}\mathbb{E}^{\zeta_n}\bigg[ \mathbf{1}\big\lbrace (z,u) \in \mathcal{A}_a \text{ and }  (\tilde{z},\tilde{u}) \in \mathcal{A}_a \big\rbrace \bigg( \norm{z - \tilde{z}}_T +  \norm{u - \tilde{u}}_T \bigg)\bigg] = 0,\label{eq: int zero final}
\end{align}%The RHS of \eqref{eq: temporary L squared ineq} goes to zero as $n\to\infty$ 
thanks to \eqref{eq: zeta n zero}. Also, as a consequence of Lemma \ref{Lemma Regularity one},
\begin{align} \label{eq: nice set a bound}
\lim_{a\to\infty}\lim_{n\to\infty} \mathbb{E}^{\zeta_n}\bigg[ \mathbf{1}\big\lbrace (z,u) \notin \mathcal{A}_a \text{ or } (\tilde{z},\tilde{u}) \notin \mathcal{A}_a \big\rbrace  \bigg] = 0.
\end{align}
It follows from Hypothesis \ref{Hypothesis Distribution of Positions}, and making use of the fact that $d_{\mathcal{E}}$ is uniformly upperbounded, that
\begin{align}
\lim_{n\to\infty} \mathbb{E}^{\zeta_n}\bigg[ 1 \wedge \bigg( d_{\mathcal{E}}(\theta,\tilde{\theta}) + d_{\mathcal{E}}(\eta,\tilde{\eta}) \bigg) \bigg] 
\leq \rm{Const}\lim_{n\to\infty} \mathbb{E}^{\zeta_n}\bigg[   d_{\mathcal{E}}(\theta,\tilde{\theta}) + d_{\mathcal{E}}(\eta,\tilde{\eta})   \bigg] = 0. \label{eq: limit of empirical distribution of weights}
\end{align}
Since
\begin{multline}
 \mathbb{E}^{\zeta_n}\bigg[ 1 \wedge \bigg( d_{\mathcal{E}}(\theta,\tilde{\theta}) + d_{\mathcal{E}}(\eta,\tilde{\eta}) +  \norm{z - \tilde{z}}_T +  \norm{u - \tilde{u}}_T \bigg)^2 \bigg] \\
 \leq  \mathbb{E}^{\zeta_n}\bigg[ 1 \wedge \bigg( \norm{z - \tilde{z}}_T +  \norm{u - \tilde{u}}_T \bigg) +
 1 \wedge\bigg( d_{\mathcal{E}}(\theta,\tilde{\theta}) + d_{\mathcal{E}}(\eta,\tilde{\eta}) \bigg) \bigg], 
\end{multline}
we can conclude from \eqref{eq: zeta n zero}, \eqref{eq: nice set a bound}, \eqref{eq: int zero final} and \eqref{eq: limit of empirical distribution of weights} that
\begin{align}
    \lim_{n\to \infty}d_{W}\big( \hat{\mu}^n  , \mu \big) = 0 .
\end{align}
%Thanks to the Triangle Inequality, it suffices to prove the following two inequalities,
%\begin{align}
%\lim_{n\to \infty}d_{W}\big( \hat{\mu}^n  , \bar{\mu}^n \big) %&= 0 \text{ and } \\
%\lim_{n\to \infty}d_{W}\big( \mu  , \bar{\mu}^n \big) &= 0.
%\end{align}
%Suppose for a contradiction that there exists a subsequence $\lbrace n(i)\rbrace_{i \geq 1}$ such that
%\begin{align}
%    \lim_{i\to \infty}d_{W}\big( \hat{\mu}^{n_i}  , \mu \big) & := \delta > 0.
%\end{align}
%Thanks to Lemma \ref{Lemma Regularity one}, 
%\end{proof}
%Our first step is to show that the empirical measure for the system with more homogeneous interactions must concentrate at the same value:
%Theorem \ref{Theorem main} will thus follow from Lemma \ref{Lemma Main Comparison Disordered Graph}, as well as the following Lemma.
 
\end{proof}

\subsection{Comparing to the system with homogeneous coupling}\label{Section Averaged Coupling}
The goal of this section is to prove Lemma \ref{Lemma intermediate 1}. Employing the `obvious coupling' by matching indices, we have that
\begin{align}
d_{L^2}\big( \hat{\mu}^n  , \bar{\mu}^n \big)^2 \leq &
n^{-2} \sum_{j,k\in \mathbb{N}_n} \int_0^T \bigg( \| \bar{u}^j_t - u^j_t \|^2 +  \| \bar{z}^{jk}_t - z^{jk}_t \|^2 \bigg) dt \nonumber \\
\leq &\frac{T}{n}\sup_{t\leq T} \sum_{j\in \mathbb{N}_n} \| \bar{u}^j_t - u^j_t \|^2 +  \frac{T}{n^2} \sup_{t\leq T}\sum_{j,k\in \mathbb{N}_n} \| \bar{z}^{jk}_t - z^{jk}_t \|^2.
\end{align}
In order that Lemma \ref{Lemma intermediate 1} holds, it thus suffices that we prove the following Lemma.
 \begin{lemma} \label{Lemma one converegence}
With unit probability,
\begin{align}
\lim_{n\to\infty} n^{-1} \sup_{t\leq T} \sum_{j\in \mathbb{N}_n} \| \bar{u}^j_t - u^j_t \|^2 &= 0 \\
\lim_{n\to\infty} n^{-2} \sup_{t\leq T} \sum_{j,k\in \mathbb{N}_n} \| \bar{z}^{jk}_t - z^{jk}_t \|^2 &= 0 .
\end{align}
\end{lemma}
\begin{proof}
Define
\begin{align}
\zeta_t &= \bigg( n^{-1} \sum_{j\in \mathbb{N}_n} \| \bar{u}^j_t - u^j_t \|^2 \bigg)^{1/2} \\
\eta_t &= \bigg(n^{-2} \sum_{j,k\in \mathbb{N}_n} \| \bar{z}^{jk}_t - z^{jk}_t \|^2\bigg)^{1/2}
\end{align}
Thanks to Ito's Lemma, 
\begin{multline}
d\zeta_t^2 = n^{-1} \sum_{j\in \mathbb{N}_n}\bigg\lbrace 2 \bigg\langle \bar{u}^j_t - u^j_t , f(\bar{u}^j_t) - f(u^j_t) +\Gamma^{n,j}_t \bigg\rangle dt\\ + 2\bigg\langle \bar{u}^j_t - u^j_t , \big( \sigma(x^j_n, \bar{u}^j_t) - \sigma(x^j_n, u^j_t) \big)dW^j_t \bigg\rangle \\
+\rm{tr}\bigg( \big( \sigma^T(x^j_n, \bar{u}^j_t) - \sigma^T(x^j_n, u^j_t) \big)\big( \sigma(x^j_n, \bar{u}^j_t) - \sigma(x^j_n, u^j_t) \big) \bigg) dt\bigg\rbrace ,
\end{multline}
where
\begin{align}
\Gamma^{n,j}_t =& \tilde{\Gamma}^{n,j}_t + \grave{\Gamma}^{n,j}_t \text{ where } \\
\tilde{\Gamma}^{n,j}_t =& n^{-1}\sum_{k\in \mathbb{N}_n}   \big(   \mathcal{K}(x^j_n,x^k_n) - \varphi_n^{-1} K_n^{jk}  \big) F (x^j_n , x^k_n, \bar{u}^j_t, \bar{z}^{jk}_t) \\
\grave{\Gamma}^{n,j}_t =&  n^{-1}\sum_{k\in \mathbb{N}_n}   \varphi_n^{-1} K_n^{jk} \big(  F (x^j_n , x^k_n, \bar{u}^j_t, \bar{z}^{jk}_t) - F (x^j_n , x^k_n, u^j_t, z^{jk}_t)\big)
\end{align}
Thanks to Hypothesis \ref{Hypothesis Graphon}, and noting that $F$ is uniformly upperbounded (by assumption), there must exist a constant $C> 0$ such that
\begin{align}
n^{-1} \sum_{j\in \mathbb{N}_n} \|\grave{\Gamma}^{n,j}_t \|^2  &\leq \frac{C}{n^2} \sum_{j,k\in \mathbb{N}_n} \big\|   F (x^j_n , x^k_n, \bar{u}^j_t, \bar{z}^{jk}_t) - F (x^j_n , x^k_n, u^j_t, z^{jk}_t)\big\|^2 \\
n^{-1} \sum_{j\in \mathbb{N}_n} \|\tilde{\Gamma}^{n,j}_t \|^2  &\leq C \delta_n %\frac{\delta_n}{n^2} \sum_{j,k\in \mathbb{N}_n} \big\|   F (x^j_n , x^k_n, \bar{u}^j_t, \bar{z}^{jk}_t) \|^2 .
\end{align}
where $\big\lbrace \delta_n\big\rbrace_{n\geq 1}$ are constants defined in Hypothesis \ref{Hypothesis Graphon} such that
\begin{align}
\lim_{n\to\infty} \delta_n = 0.
\end{align}
Thanks to Hypothesis \ref{Hypothesis Lipschitz Functions} and the Cauchy-Schwarz Inequality, there must exist a constant $C > 0$ such that
\begin{align}
d\zeta^2_t &\leq C\big( \zeta^2_t + \eta^2_t + \sqrt{\delta_n} \zeta_t \big) dt + dw_t 
\end{align}
where
\begin{align}
dw_t &= \frac{2}{n} \sum_{j\in \mathbb{N}_n} \bigg\langle u^j_t - \bar{u}^j_t , \big( \sigma(x^j_n , u^j_t) - \sigma(x^j_n , \bar{u}^j_t) \big) dW^j_t \bigg\rangle.
\end{align}
We analogously find that for some constant $C > 0$,
\begin{align}
d\eta^2_t &\leq C\big( \zeta^2_t + \eta^2_t \big) dt+ d\tilde{w}_t \text{ where} \\
d\tilde{w}_t &= \frac{2}{n^2} \sum_{j,k\in \mathbb{N}_n} \bigg\langle z^{jk}_t - \bar{z}^{jk}_t , \big( \gamma(x^j_n , x^k_n, z^{jk}_t  , u^{k}_t ) - \gamma(x^j_n ,x^k_n, \bar{z}^{jk}_t ,  \bar{u}^k_t ) \big) dW^{jk}_t \bigg\rangle .
\end{align}
Taking square roots, and using the fact that $\frac{d^2}{dx^2}\sqrt{x} < 0$, another application of Ito's Lemma implies that at long as $\zeta_t > 0$ and $\eta_t > 0$,
\begin{align}
d\zeta_t &\leq C\big( \max \big\lbrace \zeta_t , \eta_t \big\rbrace + \sqrt{\delta_n} \big) dt + d\hat{w}_t \\
d\eta_t &\leq C \max \big\lbrace \zeta_t , \eta_t \big\rbrace dt+ d\breve{w}_t,
\end{align}
where
\begin{align}
d\hat{w}_t =& \frac{1}{2\zeta_t} dw_t \\
d\breve{w}_t =&\frac{1}{2\eta_t} d\tilde{w}_t.
\end{align}
Define the quadratic variation of $\hat{w}_t$ to be 
\begin{align}
n^{-1} \int_0^t \hat{q}(s) ds 
\end{align}
and define the quadratic variation of $\tilde{w}_t$ to be
\begin{align}
n^{-1} \int_0^t \tilde{q}(s) ds .
\end{align}
Since $\sigma$ and $\gamma$ are assumed to be bounded and uniformly Lipschitz, $\hat{q}(s)$ and $\tilde{q}(s)$ must be uniformly bounded (thanks to the Cauchy-Schwarz Inequality). In particular, these uniform bounds hold for arbitrarily small $\zeta_t$ and $\eta_t$.

Using the time-rescaled representation of a stochastic integral \cite{Karatzas1991}, we find that there is a constant $C > 0$ such that if $W(t)$ is a standard Brownian Motion,
\begin{align}
\mathbb{P}\big( \sup_{t\leq T} \big| \hat{w}_t \big| \geq \epsilon \big) \leq & \mathbb{P}\bigg( \sup_{t\leq T} \big| W\big( Ct/n \big) \big| \geq \epsilon \bigg) \nonumber \\
\leq & \exp\big( - C_2 n \big),
\end{align}
where $C_2> 0$ is another constant, using standard properties of Brownian Motion \cite{morters2010brownian}. It thus follows from the Borel-Cantelli Lemma that $\mathbb{P}$-almost-surely,
\begin{align}
\lim_{n\to\infty} \sup_{t\leq T} | \hat{w}_t | = 0.
\end{align}
We analogously find that
\begin{align}
\lim_{n\to\infty}\sup_{t\leq T}  | \tilde{w}(t) | &= 0 .
\end{align}
Now Gronwall's Inequality implies that
\begin{align}
 \sup_{t\leq T}\big\lbrace \zeta_t + \eta_t \big\rbrace \leq \exp(CT)\bigg\lbrace \delta_n +  \sup_{t\leq T}  | \hat{w}(t) | +\sup_{t\leq T}  | \tilde{w}(t) | \bigg\rbrace .
\end{align}
We can thus conclude that $\mathbb{P}$-almost-surely,
\begin{align}
\lim_{n\to\infty}  \sup_{t\leq T}\big\lbrace \zeta_t + \eta_t \big\rbrace = 0.
\end{align}
\end{proof}

\subsection{Coupling Argument} \label{Section Coupling}
The main focus of this section is to prove Lemma \ref{Lemma intermediate 6}, i.e. we prove that
\begin{align}
\lim_{n\to\infty} d_{L^2}\big( \bar{\mu}^n, \mu \big) = 0.
\end{align}
To do this we establish the intermediate Lemmas \ref{Lemma grave mu convergence} and \ref{Lemma grave mu bar mu}. We start with Lemma \ref{Lemma grave mu bar mu}, i.e. we must show that $\mathbb{P}$-almost-surely,
\begin{align}
\lim_{n\to\infty} \sup_{j\in I_n} d_{L^2}\big( \grave{\mu}^{n,j} , \mu_{x^j_n} \big) &= 0 \label{Eq: intermediate grave mu 1}\\
\lim_{n\to\infty} d_{L^2}\big( \grave{\mu}^n , \mu \big) &= 0 .\label{Eq: intermediate grave mu 2}
\end{align}
In fact \eqref{Eq: intermediate grave mu 2} is an immediate consequence of \eqref{Eq: intermediate grave mu 1}. One can construct a coupling such that
\begin{align}
d_{L^2}\big( \grave{\mu}^n , \mu \big)^2 \leq 2d_W\big( \hat{\mu}^n(x_n) , \mu \big)^2 +2 \sup_{j\in \mathbb{N}_n} d_{L^2}\big( \grave{\mu}^{n,j} , \mu_{x^j_n} \big)^2. \label{eq: coupling decomposed empirical measures}
\end{align}
The RHS of \eqref{eq: coupling decomposed empirical measures} goes to zero as $n\to\infty$ thanks to Hypothesis \ref{Hypothesis Distribution of Positions} and
\eqref{Eq: intermediate grave mu 1}. 

It thus remains to prove \ref{Eq: intermediate grave mu 1}. Observe that $\big\lbrace \big( \grave{z}^{jk}_{[0,T]} , \grave{u}^k_{[0,T]} \big) \big\rbrace_{k\in \mathbb{N}_n}$ constitutes $n$ independent $\mathcal{C}\big( [0,T], \mathbb{R}^{c+d} \big)$-valued random variables, and that the law of $\big( \grave{z}^{jk}_{[0,T]} , \grave{u}^k_{[0,T]} \big)$ is $\mu_{x^j_n x^k_n}$. It thus follows from the conditional Sanov Theorem \cite{Dembo1998,csiszar1984sanov}  that
\begin{align}
\lsup{n} n^{-1} \log \sup_{j\in \mathbb{N}_n}\mathbb{P}\bigg( d_W\big( \grave{\mu}^{n,j} , \mu_{x^j_n} \big) \geq \epsilon \bigg) < 0. \label{eq: temporary indepednent Sanov}
\end{align}
Since $|\mathbb{N}_n| = n$, employing a union-of-events bound we find that
\begin{align}
\lsup{n} n^{-1} \log \mathbb{P}\bigg( \sup_{j\in \mathbb{N}_n} d_W\big( \grave{\mu}^{n,j} , \mu_{x^j_n} \big) \geq \epsilon \bigg) 
\leq & \lsup{n} n^{-1} \log  \bigg\lbrace n \sup_{j\in \mathbb{N}_n} \mathbb{P}\bigg(  d_W\big( \grave{\mu}^{n,j} , \mu_{x^j_n} \big) \geq \epsilon \bigg) \bigg\rbrace \nonumber\\
<& 0,
\end{align}
 thanks to \eqref{eq: temporary indepednent Sanov}. We can thus conclude that \ref{Eq: intermediate grave mu 1} holds.

We must next show that $\mathbb{P}$-almost-surely,
\begin{align}
\lim_{n\to\infty} \sup_{j\in \mathbb{N}_n} d_{L^2}\big( \bar{\mu}^{n,j} , \grave{\mu}^{n,j} \big) &= 0. \label{eq: to show last section 1}\\
\lim_{n\to\infty} d_{L^2}\big( \bar{\mu}^n , \grave{\mu}^n \big) &= 0.\label{eq: to show last section 2}
\end{align} 
Observe that, assuming that $\big( \bar{z}, \bar{u} \big)$ and $\big( \grave{z}, \grave{u} \big)$ are driven by the same Brownian Motions (this is possible because they are both strong solutions \cite{Karatzas1991}),
\begin{align}
d_{L^2}\big( \bar{\mu}^n  , \grave{\mu}^n \big)^2 \leq n^{-2} \sum_{j,k \in I_n} \int_0^T \| \bar{z}^{jk}_t - \grave{z}^{jk}_t \|^2 dt  + n^{-1}\sum_{j\in I_n} \int_0^T \| \bar{u}^j_t - \grave{u}^j_t \|^2 dt 
\end{align}
Define
\begin{align}
\bar{\mu}^{n,j}_t &= \frac{1}{n} \sum_{j\in \mathbb{N}_n} \delta_{x^k_n , \bar{z}^{jk}_t , \bar{u}^k_t} \in \mathcal{P}\big( \mathcal{E} \times \mathbb{R}^d \times \mathbb{R}^c \big) \\
\grave{\mu}^{n,j}_t &= \frac{1}{n} \sum_{j\in \mathbb{N}_n} \delta_{x^k_n , \grave{z}^{jk}_t , \grave{u}^k_t} \in \mathcal{P}\big( \mathcal{E} \times \mathbb{R}^d \times \mathbb{R}^c \big).
\end{align}
Recall that $\mu_{x,t} \in \mathcal{P}\big( \mathcal{E} \times \mathbb{R}^d \times \mathbb{R}^c \big)$ is the marginal distribution of the limiting measure $\mu$ (defined in \eqref{eq: mu specification 1}) at $x,t$.

Thanks to Ito's Lemma, and employing the one-sided Lipschitz assumption in Hypothesis \ref{Hypothesis Lipschitz Functions},
\begin{multline}
d  \|  \bar{u}^j_t - \grave{u}^j_t \|^2 \leq \bigg\lbrace 2C  \|  \bar{u}^j_t - \grave{u}^j_t \|^2 + 2\bigg\langle  \bar{u}^j_t - \grave{u}^j_t ,  H (x^j_n , \bar{u}^j_t ,\bar{\mu}^{n,j}_{t}) -  H (x^j_n , \grave{u}^j_t ,\mu_{x^j_n,t}) \bigg\rangle \\+ \rm{tr}\bigg\lbrace \big( \sigma(x^j_n, \bar{u}^j_t) - \sigma(x^j_n, \grave{u}^j_t) \big)^T \big( \sigma(x^j_n, \bar{u}^j_t) - \sigma(x^j_n, \grave{u}^j_t) \big) \bigg\rbrace \bigg\rbrace dt + 2\bigg\langle  \bar{u}^j_t - \grave{u}^j_t , \big( \sigma(x^j_n, \bar{u}^j_t) - \sigma(x^j_n, \grave{u}^j_t) \big) dW^j_t \bigg\rangle \label{eq: u j t SDE difference}
\end{multline}
and 
\begin{multline}
d  \|  \bar{z}^{jk}_t - \grave{z}^{jk}_t \|^2 \leq \bigg\lbrace 2C  \|  \bar{z}^{jk}_t - \grave{z}^{jk}_t \|^2   \\ + \rm{tr}\bigg\lbrace \big( \gamma(x^j_n, x^k_n ,  \bar{z}^{jk}_t , \bar{u}^k_t ) - \gamma(x^j_n,x^k_n ,  \grave{z}^{jk}_t , \grave{u}^k_t) \big)^T \big( \gamma(x^j_n, x^k_n ,  \bar{z}^{jk}_t , \bar{u}^j_t) - \gamma(x^j_n, x^k_n, \grave{z}^{jk}_t, \grave{u}^k_t) \big) \bigg\rbrace \bigg\rbrace dt \\+ 2\bigg\langle  \bar{z}^{jk}_t - \grave{z}^{jk}_t , \big( \gamma(x^j_n, x^k_n, \bar{z}^{jk}_t ,  \bar{u}^k_t) - \gamma(x^j_n, x^k_n, \grave{z}^{jk}_t, \grave{u}^k_t) \big) dW^{jk}_t \bigg\rangle .\label{eq: z j t SDE difference}
\end{multline}
Thanks to the Triangle Inequality,
\begin{multline*}
 \big\| H (x^j_n , \bar{u}^j_t ,\bar{\mu}^{n,j}_{t}) -  H (x^j_n , \grave{u}^j_t ,\mu_{x^j_n,t}) \big\| \leq  \big\| H (x^j_n , \bar{u}^j_t ,\mu_{x^j_n,t} )  -  H (x^j_n , \grave{u}^j_t ,\mu_{x^j_n,t}) \big\| \\+  \big\| H (x^j_n , \bar{u}^j_t ,\bar{\mu}^{n,j}_{t}) -  H (x^j_n , \bar{u}^j_t ,\grave{\mu}^{n,j}_{t}) \big\| + \big\| H (x^j_n , \bar{u}^j_t ,\grave{\mu}^{n,j}_{t}) -  H (x^j_n , \bar{u}^j_t ,\mu_{x^j_n,t}) \big\|
\end{multline*}
Since $H$ is Lipschitz in its arguments, it follows from \eqref{Eq: intermediate grave mu 1} that $\mathbb{P}$-almost-surely,
\begin{align}
\lim_{n\to\infty} \sup_{j\in I_n} \sup_{t\leq T} \big\| H (x^j_n , \bar{u}^j_t ,\grave{\mu}^{n,j}_{t}) -  H (x^j_n , \bar{u}^j_t ,\mu_{x^j_n,t}) \big\| \leq C \lim_{n\to\infty} \epsilon_n= 0   
\end{align}
where
\[
\epsilon_n =  \sup_{j\in I_n} \sup_{t\leq T} d_W\big( \grave{\mu}^{n,j}_{t} , \mu_{x^j_n,t} \big) ,
\]
and $d_W$ is the Wasserstein $1$-distance on $\mathcal{P}\big( \mathcal{E} \times \mathbb{R}^d \times \mathbb{R}^c\big)$. Hypothesis \ref{Hypothesis Lipschitz Functions} implies that there exists a constant $C> 0$ such that
\begin{align}
 \big\| H (x^j_n , \bar{u}^j_t ,\mu_{x^j_n,t} )  -  H (x^j_n , \grave{u}^j_t ,\mu_{x^j_n,t}) \big\| &\leq C \big\| \bar{u}^j_t - \grave{u}^j_t \big\| \\
  \big\| H (x^j_n , \bar{u}^j_t ,\bar{\mu}^{n,j}_{t}) -  H (x^j_n , \bar{u}^j_t ,\grave{\mu}^{n,j}_{t}) \big\| &\leq \frac{C}{n} \sum_{k\in \mathbb{N}_n} \big\| \bar{z}^{jk}_t - \grave{z}^{jk}_t \big\|  .
\end{align}
Write
\begin{align}
A^j_t =& n^{-1}\sum_{k\in \mathbb{N}_n} \|  \bar{z}^{jk}_t - \grave{z}^{jk}_t \|^2  \\
B_t =& n^{-1}\sum_{j\in \mathbb{N}_n}  \|  \bar{u}^j_t - \grave{u}^j_t \|^2 ,
\end{align}
and notice that for all $j\in \mathbb{N}_n$,
\begin{align}
d_{L^2}\big( \grave{\mu}^{n,j} , \bar{\mu}^{n,j} \big)^2 \leq & T \sup_{t\leq T} \big\lbrace A^j_t + B_t \big\rbrace \label{eq: last part 1} \\
d_{L^2}\big( \grave{\mu}^n , \bar{\mu}^n \big)^2 \leq &  T \sup_{t\leq T} \bigg\lbrace n^{-1}\sum_{j\in \mathbb{N}_n}A^j_t + B_t \bigg\rbrace .\label{eq: last part 2}
\end{align}
Applying the Cauchy-Schwarz Inequality several times to \eqref{eq: u j t SDE difference}-\eqref{eq: z j t SDE difference}, we find that there exists a constant $\tilde{C} > 0$ such that
\begin{multline}
B_t \leq \tilde{C} \int_0^t \big( B_s + \epsilon_n \sqrt{B_s} + \sup_{j\in \mathbb{N}_n} A^j_s \big) ds \\+ \frac{2}{n} \sum_{k\in \mathbb{N}_n} \bigg\langle  \bar{u}^k_t - \grave{u}^k_t , \big( \sigma(x^k_n, \bar{u}^k_t) - \sigma(x^k_n, \grave{u}^k_t) \big) dW^k_t \bigg\rangle 
\end{multline}
\begin{multline}
A^j_t \leq \tilde{C} \int_0^t \big( A^j_s + B_s \big) ds\\ + \frac{2}{n} \sum_{k\in \mathbb{N}_n} \int_0^t \bigg\langle  \bar{z}^{jk}_s- \grave{z}^{jk}_s, \big( \gamma(x^j_n, x^k_n, \bar{z}^{jk}_s ,  \bar{u}^k_s) - \gamma(x^j_n, x^k_n, \grave{z}^{jk}_s, \grave{u}^k_s) \big) dW^{jk}_s \bigg\rangle 
\end{multline}
Define $Q^j_t$ to be the quadratic variation of the stochastic integral in the definition of $A^j_t$, i.e.
\begin{align}
Q^j_t = \frac{4}{n^2} \sum_{k\in \mathbb{N}_n} \int_0^t \bigg\| \big( \gamma^T(x^j_n, x^k_n, \bar{z}^{jk}_s ,  \bar{u}^k_s) - \gamma^T(x^j_n, x^k_n, \grave{z}^{jk}_s, \grave{u}^k_s) \big) \big( \bar{z}^{jk}_s- \grave{z}^{jk}_s \big) \bigg\|^2 ds,
\end{align}
and define $Q_t$ to be the quadratic variation of the stochastic integral in the definition of $B_t$, i.e.
\begin{align}
Q_t = \frac{4}{n^2} \sum_{k\in \mathbb{N}_n} \int_0^t \bigg\| \big( \sigma^T(x^k_n, \bar{u}^k_s) - \sigma^T(x^k_n, \grave{u}^k_s) \big) \big(  \bar{u}^k_s- \grave{u}^k_s  \big) \bigg\|^2 ds.
\end{align}
Let $w^j(t)$ and $w(t)$ be functions such that
\begin{align}
w^j\big( Q^j_t \big) =&  \frac{2}{n} \sum_{k\in \mathbb{N}_n} \int_0^t \bigg\langle  \bar{z}^{jk}_s- \grave{z}^{jk}_s, \big( \gamma(x^j_n, x^k_n, \bar{z}^{jk}_s ,  \bar{u}^k_s) - \gamma(x^j_n, x^k_n, \grave{z}^{jk}_s, \grave{u}^k_s) \big) dW^{jk}_s \bigg\rangle \\
w\big( Q_t \big) =&  \frac{2}{n} \sum_{k\in \mathbb{N}_n} \bigg\langle  \bar{u}^k_t - \grave{u}^k_t , \big( \sigma(x^k_n, \bar{u}^k_t) - \sigma(x^k_n, \grave{u}^k_t) \big) dW^k_t \bigg\rangle
\end{align}
The time-rescaled representation of stochastic integrals imply that $\lbrace w^j(t) , w(t) \rbrace_{j\in \mathbb{N}_n}$ are independent Brownian Motions \cite{Karatzas1991}. Write
\begin{align}
D_t = B_t + \sup_{j\in \mathbb{N}_n} A^j_t.
\end{align}
We now find that there is a constant $\hat{C}$ such that for all $t\leq T$, it holds that
\begin{align}
D_t \leq \hat{C} \int_0^t D_s ds + \hat{C} \epsilon_n + \sup_{j\in \mathbb{N}_n} \sup_{s\leq t} w^j\big( Q^j_s \big) + \sup_{s\leq t} w\big( Q_s \big) ,
\end{align}
and Gronwall's Inequality thus implies that
\begin{align}
\sup_{t\leq T} D_t \leq \exp\big( T \hat{C} \big) \bigg\lbrace \hat{C} \epsilon_n +  \sup_{j\in \mathbb{N}_n} \sup_{s\leq t} w^j\big( Q^j_s \big) + \sup_{s\leq t} w\big( Q_s \big) \bigg\rbrace .
\end{align}
Define 
\begin{align}
\tau = \inf\bigg\lbrace t\leq T \; : \; B_t = 1 \text{ or } \sup_{j\in \mathbb{N}_n} A^j_t = 1 \bigg\rbrace.
\end{align}
Our model assumptions imply that there exists a constant $\bar{C} > 0$ such that for all $t\leq \tau$, it holds that
\begin{align}
Q^j_t \leq \frac{\bar{C}}{n} \text{ and } Q_t \leq \frac{\bar{C}}{n}.
\end{align}
Employing properties of Brownian Motion \cite{morters2010brownian}, we find that there exists a constant $\grave{C} > 0$ such that
\begin{align}
\mathbb{P}\bigg( \sup_{s \leq T \wedge \tau } w^j\big( Q^j_s \big) \geq n^{-1/4} \bigg) \leq
\mathbb{P}\bigg( \sup_{s \leq T \wedge \tau} w^j\big( s \bar{C} / n \big) \geq n^{-1/4} \bigg)
\leq \exp\big( - \grave{C} n^{1/2} \big).
\end{align}
Using a union of event bounds, we find that
\begin{align}
\mathbb{P}\bigg(  \sup_{j \in \mathbb{N}_n} \sup_{s  \leq T \wedge \tau } w^j\big( Q^j_s \big) \geq n^{-1/4} \bigg) \leq n  \exp\big( - \grave{C} n^{1/2} \big).
\end{align}
Similarly, 
\begin{align}
\mathbb{P}\bigg(   \sup_{s \leq T \wedge \tau } w\big( Q_s \big) \geq n^{-1/4} \bigg) \leq    \exp\big( - \grave{C} n^{1/2} \big).
\end{align}
It thus follows from the Borel-Cantelli Lemma that $\mathbb{P}$-almost-surely,
\begin{align}
\lim_{n\to\infty} n^{1/4} \sup_{s\leq T} w\big( Q_s \big) < \infty \\
\lim_{n\to\infty} n^{1/4} \sup_{j\in \mathbb{N}_n}\sup_{s\leq T} w^j\big( Q^j_s \big) < \infty .
\end{align}
We can therefore conclude that $\mathbb{P}$-almost-surely,
\begin{align}
\lim_{n\to\infty} \sup_{t\leq T} D_t = 0.
\end{align}
In light of \eqref{eq: last part 1} and \eqref{eq: last part 2}, we can therefore
conclude that \eqref{eq: to show last section 1} and \eqref{eq: to show last section
2} both hold.

\section{Acknowledgements:} We thank Axel Hutt for a very helpful discussion. James MacLaurin acknowledges support from NSF-DMS 2511615 and a grant from the Simons Foundation MPS-TSM-00713975 J.M.  This research was supported in part by grants from the NSF (DMS-2235451) and Simons Foundation (MP-TMPS-00005320) to the NSF-Simons National Institute for Theory and Mathematics in Biology (NITMB).

\appendix

\bibliographystyle{plain}
\bibliography{bib3,adaptiveadditional}

\section{Gaussian Reduction} \label{Section Gaussian Reduction}
If we make some simplifying assumptions, then the limiting law $\mu$ (defined in \eqref{eq: mu specification 1}) is Gaussian, and can be reduced to equations for the mean and covariance. On making this reduction, we obtain a `neural field' consisting of two variables: the means of the neural activity and synaptic activity, and the covariance functions. This reduction is helpful because it renders the search for patterns, oscillations and other coherent phenomena more tractable \cite{avitabile2026neural}. See our extended application in Section \ref{Section Applications}.

In more detail, our additional assumptions in this section are that for% $\alpha \in \lbrace + , - \rbrace$,
\begin{align}
f(s) &:= Ls \\
 G (\theta,\eta,u,s) &:= M_{\theta \eta}u+ Q_{\theta \eta}s \\
\gamma (\theta,\eta,u,s) &:= \tilde{\gamma}_{\theta \eta} \\
\sigma(\theta,z) &:= \tilde{\sigma}_{\theta}
\end{align}
where $L$ is a matrix, $\tilde{F}$ is globally Lipschitz, and $(\theta,\eta) \mapsto M_{\theta\eta} \in \mathbb{R}^{c\times d}$ is continuous and $(\theta,\eta) \mapsto M_{\theta\eta} \in \mathbb{R}^{c\times c}$ is continuous, $(\theta,\eta) \mapsto Q_{\theta\eta} \in \mathbb{R}$ is continuous and $(\theta,\eta) \mapsto \tilde{\gamma}_{\theta\eta} \in \mathbb{R}^{c\times d}$ is also continuous. We also assume that the initial conditions are Gaussian. More precisely, we assume that there exist continuous functions $(\theta,\eta) \mapsto m_{\theta\eta , 0} \in \mathbb{R}^c $, $\theta\mapsto \ell_{\theta,0} \in \mathbb{R}^d$, and $(\theta,\eta) \mapsto \mathfrak{V}_{\theta\eta} \in \mathbb{R}^{c+d \times c+d}$ such that for any measurable subsets $A_1,A_2 \subseteq \mathcal{E}$, $B_1 \subseteq \mathbb{R}^c$ and $B_2 \subseteq \mathbb{R}^d$,
\begin{align}
\mu_0\big( A_1 \times A_2 \times B_1 \times B_t \big) =  \int_{\mathcal{E}}  \int_{\mathcal{E}} \int_{B_1 \times B_2}  \rho_{\theta\eta,0}(J,y) dJ dy d\mu_{\mathcal{E}}(\theta)d\mu_{\mathcal{E}}(\eta).
\end{align}
Here $\rho_{\theta\eta,0}(J,y) \in \mathbb{R}^+$ is the density of a Gaussian distribution, with mean vector $\big( m_{\theta\eta,0} , \ell_{\eta,0} \big) \in \mathbb{R}^{c+d} $ and covariance matrix $\mathfrak{V}_{\theta\eta}$.

These assumptions will ensure that $\mu_t$ is Gaussian. We now specify what the limiting mean and variance is. For $\theta,\eta \in \mathcal{E}$, define $m_{s}(t,\theta\eta) \in \mathbb{R}^c$, $m_u(t,\theta) \in \mathbb{R}^d$,  $V_{ss}(t,\theta,\eta) \in \mathbb{R}^{c\times c}$,  $V_{u}(t,\theta) \in \mathbb{R}^{d\times d}$ and $V_{us}(t,\theta,\eta) \in \mathbb{R}^{d \times c}$ to be such that
\begin{align}
m_z(t,\theta,\eta) =& \int_{\mathbb{R}^d} \int_{\mathbb{R}^c} z p_{\theta\eta,t}(K,y) dz dy \\
m_u(t,\theta) =&  \int_{\mathbb{R}^d} \int_{\mathbb{R}^c} y p_{\theta\eta,t}(K,y) dK dy \\
V_{zz}(t,\theta,\eta)=& \int_{\mathbb{R}^d} \int_{\mathbb{R}^c} \big(z -m_z(t,\theta,\eta)  \big) \big( z- m_z(t,\theta,\eta)  \big)^T  p_{\theta\eta,t}(z,y) dz dy \\
V_{uu}(t,\theta) =&  \int_{\mathbb{R}^d} \int_{\mathbb{R}^c} \big( y -m_u(t,\theta)  \big) \big( y -m_u(t,\theta) \big)^T  p_{\theta\eta,t}(z,y) dz dy \\
V_{uz}(t,\theta,\eta )=& \int_{\mathbb{R}^d} \int_{\mathbb{R}^c} \big( y - m_u(t,\theta) \big) \big( z - m_z(t,\theta,\eta)  \big)^T  p_{\theta\eta,t}(z,y) dz dy 
\end{align}
\begin{lemma}
The means and variances constitute the following family of Ordinary Differential Equations (with mean-field coupling): for every $\theta,\eta\in\mathcal{E}$
\begin{align}
\frac{d}{dt} m_s(t,\theta,\eta)  =& Q_{\theta \eta} m_s(\theta,\eta,t)  +  M_{\theta \eta}m_u(\theta,t)  \\ 
\frac{d}{dt}m_u(\theta,t) =&   Lm_u(\theta,t)  + R_{\theta}(t)  \\
\frac{d}{dt} V_{ss}(t,\theta,\eta) =& Q_{\theta \eta} V_{\theta\eta}(t) + V_{\theta\eta}(t) Q_{\theta \eta}^T + \tilde{\gamma}_{\theta\eta} \tilde{\gamma}^T_{\theta\eta} \\
\frac{d}{dt} V_{us}(t,\theta,\eta)  =& L V_{us}(t,\theta,\eta) + V_{us}(t,\theta,\eta) Q_{\theta \eta}^T +  V_{uu}(t,\theta) Q_{\theta \eta}^T \\
\frac{d}{dt}  V_{uu}(t,\theta)=& L V_{uu}(t,\theta)+ V_{uu}(t,\theta) L^T +  \tilde{\sigma}_{\theta} \tilde{\sigma}_{\theta}^T
\end{align}
Here
\begin{multline}
R_{\theta}\big(m_s, m_u , V_{ss} , V_{us} , V_{uu} \big) = \\ \int_{\mathcal{E}} \int_{\mathbb{R}^{c+d}}  \mathcal{K}(\theta,y)  F(s,y) \rho \big( m_s , m_u , V_{ss}, V_{us} , V_{uu} , s,y \big) ds dy d\mu_{\mathcal{E}}(\eta) 
\end{multline}
and $\rho_{\theta\eta,t}(J,y)$ is the Gaussian density for an $\mathbb{R}^{c+d}$-valued Gaussian variable, with mean vector 
\[
\left(\begin{array}{c}
 m_s(t,\theta,\eta)\\
m_u(t,\theta) 
\end{array}\right) \in \mathbb{R}^{c+d}
\]
and covariance matrix
\begin{equation}
\left(\begin{array}{c c}
V_{ss}(t,\theta,\eta)  &V_{su}(t,\theta,\eta ) \\
V_{us}(t,\theta,\eta ) &V_{uu}(t,\theta)
\end{array}\right) \in \mathbb{R}^{(c+d) \times (c+d)}.
\end{equation}
\end{lemma}

\end{document}